\documentclass{aimsarx}
\usepackage{amsmath}
  \usepackage{paralist}
  \usepackage{graphics} 
  \usepackage{epsfig} 
\usepackage{graphicx}  \usepackage{epstopdf}
 \usepackage[colorlinks=true]{hyperref}
\hypersetup{urlcolor=blue, citecolor=red}

\def\currentvolume{ }
 \def\currentissue{ }
  \def\currentyear{ }
   \def\currentmonth{ }
    \def\ppages{ }

   \usepackage{graphicx}
\usepackage{setspace} 
\usepackage{amssymb, pifont, color, graphics, amsmath, amssymb, mathrsfs, amsfonts, multicol, amsmath, tikz, pgfplots, graphicx, lmodern,footnote} 
\usepackage{tikz}
\usepackage[ruled,vlined]{algorithm2e}
\usepackage{hyperref,stackengine}
\hypersetup{
    colorlinks=true,
    linkcolor=blue,
    filecolor=red,      
    urlcolor=red,
}

\usepackage{mathtools}
\usepackage{esint}
\usepackage{lineno}

\makeatletter
\newcommand{\leqnomode}{\tagsleft@true\let\veqno\@@leqno}
\newcommand{\reqnomode}{\tagsleft@false\let\veqno\@@eqno}
\makeatother
\usepackage[utf8]{inputenc}

\DeclareMathOperator*{\du}{d\!}
\DeclareMathOperator*{\Div}{div}

\DeclareMathOperator{\bu}{\boldsymbol{u}}
\DeclareMathOperator{\blf}{\boldsymbol{f}}

\DeclareMathOperator{\bv}{\boldsymbol{v}}

\DeclareMathOperator{\bH}{\mathbf{H}}

\DeclareMathOperator{\bTh}{\mathbf{\Theta}}
\DeclareMathOperator{\bth}{\boldsymbol{\theta}}
\DeclareMathOperator{\bn}{\boldsymbol{n}}
\DeclareMathOperator{\bN}{\boldsymbol{N}}
\DeclareMathOperator{\bg}{\boldsymbol{g}}

\DeclareMathOperator{\bphi}{\boldsymbol{\varphi}}

\DeclareMathOperator{\btau}{\boldsymbol{\tau}}
\DeclareMathOperator{\dive}{\mathrm{div}}

\DeclareMathOperator{\Oad}{\mathcal{O}_{ad}}
\DeclareMathOperator{\ws}{\mathrel{\ensurestackMath{\stackon[1pt]{\rightharpoonup}{\scriptstyle\ast}}}}

\newcommand\irregularcircle[2]{
  \pgfextra {\pgfmathsetmacro\len{(#1)+rand*(#2)}}
  +(0:\len pt)
  \foreach \a in {10,20,...,350}{
    \pgfextra {\pgfmathsetmacro\len{(#1)+rand*(#2)}}
    -- +(\a:\len pt)
  } -- cycle
}

\makeatletter

\newenvironment{pethau}%
{\begin{list}{}%
    {%
      \setlength{\itemindent}{-10pt}%
      \setlength{\leftmargin}{20pt}%
      \setlength{\labelwidth}{.3\normalparindent}%
      \addtolength{\topsep}{-0.5\parskip}%
      \listparindent \normalparindent
      \setlength{\parsep}{\parskip}}}%
  {\end{list}}
\makeatother

\makesavenoteenv{tabular}
\makesavenoteenv{table}

\newtheorem{theorem}{Theorem}[section]

\newtheorem{lemma}[theorem]{Lemma}
\newtheorem{proposition}{Proposition}

\theoremstyle{definition}
\newtheorem{definition}[theorem]{Definition}
\newtheorem{remark}{Remark}
\newtheorem*{notation}{Notation}

\title[Maximizing Vortex via a Shape Optimization Problem] 
      {A Shape Optimization Problem Constrained with the Stokes Equations to Address Maximization of Vortices}

\author[John Sebastian Simon and Hirofumi Notsu]{}


 \email{john.simon@stu.kanazawa-u.ac.jp}
 \email{notsu@se.kanazawa-u.se.ac}

\thanks{This work is supported by JSPS KAKENHI Grant Numbers JP18H01135, JP20H01823, JP20KK0058 and JP21H04431, and JST CREST Grant Number JPMJCR2014 for {\bf HN}; and by the Japanese Government (MEXT) Scholarship for {\bf JSS}}

\thanks{$^*$ Corresponding author: john.simon@stu.kanazawa-u.ac.jp}

\begin{document}
\maketitle

\centerline{\scshape John Sebastian H. Simon$^*$}
\medskip
{\footnotesize
 \centerline{Graduate School of Natural Science and Technology}
   \centerline{Kanazawa University}
   \centerline{Kanazawa 920-1192, Japan}
} 

\medskip

\centerline{\scshape Hirofumi Notsu}
\medskip
{\footnotesize
 \centerline{ Faculty of Mathematics and Physics}
   \centerline{Kanazawa University}
   \centerline{Kanazawa 920-1192, Japan}
}

\bigskip


\begin{abstract}
We study an optimization problem that aims to determine the shape of an obstacle that is submerged in a fluid governed by the Stokes equations.  The mentioned flow takes place in a channel, which motivated the imposition of a Poiseuille-like input function on one end and a do-nothing boundary condition on the other. 
The maximization of the vorticity is addressed by the $L^2$-norm of the curl and the {\it det-grad} measure of the fluid. 
We impose a Tikhonov regularization in the form of a perimeter functional and a volume constraint to address the possibility of topological change. 
Having been able to establish the existence of an optimal shape, the first order necessary condition was formulated by utilizing the so-called rearrangement method.
Finally, numerical examples are presented by utilizing a finite element method on the governing states, and a gradient descent method for the deformation of the domain. 
On the said gradient descent method, we use two approaches to address the volume constraint: one is by utilizing the augmented Lagrangian method; and the other one is by utilizing a class of divergence-free deformation fields.
\end{abstract}

\section*{Introduction}

The study of fluid dynamics is among the most active areas in mathematics,  physics, engineering, and most recently in information theory \cite{goto2021}. An interesting area in this field is the study of turbulent flows and on controlling the emergence of such phenomena \cite{desai1994,posta2007,flinois2015}.  

Even though turbulence represents chaos, harnessing it in a controlled manner can be useful, for instance, in optimal mixing problems. Among these studies is \cite{eggl2020} where an optimal stirring strategy to maximize mixing is studied, while in \cite{mathew2007} the authors studied the best type of input functions that will provide a better mixing in the fluid.  

Recently, Goto, K.  et.al.\cite{goto2021} utilized the generation of vortices in the fluid for information processing tasks. In the said literature, the authors found out that the length of the twin-vortex affects the memory capacity in the context of physical reservoir computing.  

For these reasons, the current paper is dedicated to studying a shape optimization problem intended to maximize the turbulence of a flow governed by the two-dimensional Stokes equations. Here, the shape of an obstacle submerged in a fluid -- that flows through a channel -- is determined to maximize vorticity. Furthermore, the vorticity is quantified in two ways: one is by the curl of the velocity of the fluid; and the other is by rewarding the unfolding of complex eigenvalues of the gradient tensor of the velocity field.

To be precise, our goal is to study the following shape optimization problem
\begin{equation}
\min\{\mathcal{G}({\bu},\Omega); |\Omega| = m, \Omega\subset D \},
\label{shapeopti}
\end{equation}
where $D$ is a hold-all domain which is assumed to be a fixed bounded connected domain in $\mathbb{R}^2$, $\Omega\subset D$ is an open bounded domain, $\mathcal{G}({\bu},\Omega) : = J({\bu},\Omega) + \alpha P(\Omega),$ $J$ and $P$ are vortex and perimeter functionals given by 
\[J({\bu},\Omega) = -\int_{\Omega} \frac{\gamma_1}{2}|\nabla\times {\bu}|^2 + \gamma_2 h(\mathrm{det}(\nabla {\bu})) \du x,\ P(\Omega) = \int_{\Gamma_{\rm f}}\du s,\]
 $h$ is such that \[ h(t) = \left\{\begin{matrix} 0&\text{if } t\le 0,\\ t^3/(t^2+1)&\text{if }t>0, \end{matrix}\right. \]
$|\Omega| := \int_{\Omega}\du x$, and $\alpha,\gamma_1,\gamma_2$ and $m$ are positive constants. Here, ${\bu}$ is the velocity field governed by a fluid flowing through a channel with an obstacle (see Figure \ref{fig1} for reference). The flow of the fluid is reinforced by a divergence-free input function $\bg$ on the left end of the channel, which is an inflow boundary denoted by $\Gamma_{\rm{in}}$, whilst an outflow boundary condition is imposed to the fluid on the right end of the channel denoted by $\Gamma_{\rm{out}}$. The boundary $\Gamma_{\rm f}$  is the free surface and is the boundary of the submerged obstacle in the fluid. The remaining boundaries of the channel are wall boundaries denoted by $\Gamma_{\rm{w}}$, upon which -- together with the obstacle boundary $\Gamma_{\rm f}$ -- a no-slip boundary condition is imposed on the fluid. We refer the reader to \cite{gao2008,haslinger2005,iwata2010,pironneau2010,Schmidt2010} for shape optimization problems involving fluids among others.

\begin{figure}[h]
\centering
\beginpgfgraphicnamed{flow}
\begin{tikzpicture}
	  \coordinate (c) at (1.75,2);
      \draw[black, fill = gray, fill opacity = 0.5, semithick, even odd rule]
            (0,0) rectangle (8,4) (c) \irregularcircle{.6cm}{.3mm};
       \draw (-0.25,2) node{$\Gamma_{\rm{in}}$}  (2.3,2.65) node {$\Gamma_{\rm{f}}$} (7.5,2) node{$\Gamma_{\rm{out}}$} (3,4.2) node {$\Gamma_{\rm{w}}$} (3,-.25) node{$\Gamma_{\rm{w}}$} (6,3.5) node{${\Omega}$};
       \foreach \y in {.25cm, .5cm, .75, 1cm, 1.25cm, 1.5cm, 1.75 ,2cm, 2.25cm, 2.5cm, 2.75 , 3cm, 3.25cm, 3.5cm, 3.75 }
		\draw[->] (0pt,\y) -- (20pt,\y);
\end{tikzpicture}
\endpgfgraphicnamed
\caption{Set up of the domain.}
\label{fig1}
\end{figure}
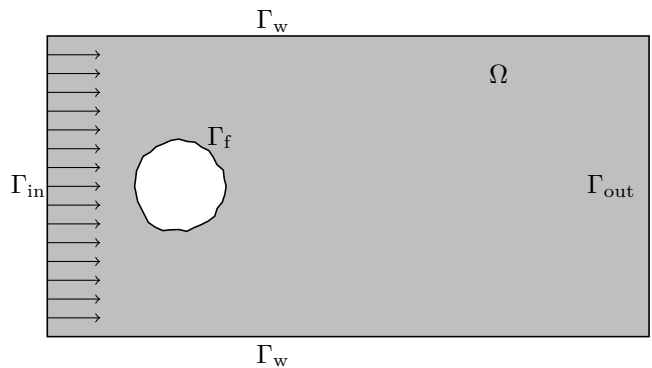
 
We also point out that usually, the curl $\nabla\times{\bu}$ is enough to quantify the vorticity of the fluid. However, as pointed out by Kasumba, K. and Kunisch, K. in \cite{kasumba2012}, the magnitude of such quantity may still be high on laminar flows.  As such, on the same literature the authors proposed the second term in the vortex functional for the quantification of rotation of fluids.  The impetus is that the rotational cores of fluids mostly occur near regions where the eigenvalues of $\nabla{\bu}$ are complex \cite{chong1990,hunt1988,jeong1995},  and that the eigenvalues are complex when $\mathrm{det}\nabla{\bu} > 0$.

{\bf Convention:} When we talk about a domain $\Omega$, we always consider it having the boundary $\partial\Omega = \overline{\Gamma}_{\rm in}\cup\overline{\Gamma}_{\rm w}\cup\overline{\Gamma}_{\rm out}\cup\overline{\Gamma}_{\rm f}$ with $\mathrm{meas}(\Gamma_{i}\cap\Gamma_j)=0$ for $i,j\in\{{\rm in,w,out,f }\}, i\ne j$, and that $\mathrm{dist}(\partial\Omega\backslash\Gamma_{\rm f},\Gamma_{\rm f}) > 0$.

One of the challenges in this problem is the non-convexity of the functional $J$, which leads into the possible non-existence of solutions for the shape optimization problem. Fortunately, we will be able to show later on that the functional $J$ is continuous with respect to domain variations. Furthermore, the Tikhonov regularization in the form of the perimeter functional $P$ is taken into consideration to circumvent the issue of possible topological changes of the domain. 

Meanwhile, for the equality constraint $|\Omega| = m$, we note that the use of an augmented Lagrangian generates smoother solutions as compared to the common Lagrangian. This is due to the regularizing effect of the quadratic augmentation. As a further consequence, this quadratic term also acts as a penalizing term for the violation of the equality constraint. Aside from the augmented Lagrangian, we also propose an analogue of the $H^1$-gradient method that satisfies the incompressibility condition to preserve the volume on each iterate of the gradient-descent method.

This paper is organized as follows: in the next section, the system that governs the fluid and the functional spaces to be used for the analysis will be introduced. The variational formulation of the governing system and the existence of its solution will also be shown in the next section.
Section \ref{sec 3} is dedicated to the analysis of the existence of an optimal shape. Here, we employ the $L^\infty$-topology of characteristic functions for convergence of deformed domains, to ensure the volume preservation.  

We shall derive the necessary condition for the optimization problem in Section \ref{sec 4}. This is done by investigating the sensitivity of the objective functional $\mathcal{G}$ with respect to shape variations. In particular, the shape variations are induced by the velocity method \cite{sokolowski1992}, and the shape sensitivity is analyzed using the rearrangement method which was formalized by Ito, K., Kunisch, K., and Peichl, G. in \cite{ito2008}.  In Section \ref{sec 5}, we provide numerical algorithms through the help of the necessary conditions formulated in Section \ref{sec 4}, and provide some illustrations of how these algorithms are implemented. To objectively observe the effects of the final shapes to fluid vortices, we look at the generation of twin-vortices in a dynamic nonlinear fluid flow. Finally, we draw conclusions and possible problems on the last section.

\section{Preliminaries}\label{sect2}
\subsection{Governing equations and necessary functional spaces}
Let $\Omega\subset D$ be open, bounded and of class $C^{1,1}$, the motion of the fluid is described by its velocity ${\bu}$ and pressure $p$ which satisfies the stationary Stokes equations given by:
\begin{align}
	\left\{
	\begin{aligned}
	\, - \nu\Delta{\bu} +  \nabla p
	&= \blf && \text{ in } \Omega,\\	
	\, \Div{\bu} &= 0 && \text{ in } \Omega,\\ 		
	\, {\bu} & = \bg && \text{ on }\Gamma_{\rm in},\\
	\, {\bu} &= 0 && \text{ on } \Gamma_{\rm{f}}\cup\Gamma_{\rm w},\\
	\, -p\bn + \nu \partial_{\bn\!}{\bu} & = 0 &&\text{ on }\Gamma_{\rm out},
	\end{aligned} \right.
	\label{state}
\end{align}
where $\nu>0$ is a constant, $\blf$ is an external force, $\bg$ is a divergence-free input function acting on the boundary $\Gamma_{\rm in}$, $\bn$ is the outward unit normal vector on $\partial\Omega$, and $\partial_{\bn}:= \bn\cdot\nabla$ is the $\bn$-directional derivative on $\partial\Omega$.
The condition ${\bu}= 0$ on $\Gamma_{\rm{f}}\cup \Gamma_{\rm w}$ corresponds to no-slip boundary condition, while the Neumann condition on $\Gamma_{\rm out}$ is the so-called {\it do-nothing} boundary condition.

The analysis will take place in the Sobolev spaces $W^{k,p}(D)$ for $D\subset\mathbb{R}^2$, $k\ge 0$ and $p\ge 1$. Note that if $p=2$, then $H^k(D)=W^{k,2}(D)$ and $H^0(D)=L^2(D)$. 
For any domain $\Omega\subset\mathbb{R}^2$, we define 
\[\mathcal{W}(\Omega) = \{\boldsymbol{\varphi} \in C^\infty(\Omega)^2 : \boldsymbol\varphi=0\text{ on a neighborhood of }\Gamma_0:=\partial \Omega\backslash \Gamma_{\rm out} \}.\]
We denote by $\bH_{\Gamma_0}^r(\Omega)$ the closure of $\mathcal{W}(\Omega)$ with respect to the space $H^r(\Omega)^2$.

We also consider the following bilinear forms $a(\cdot,\cdot)_\Omega:H^2(\Omega)^2\times H^2(\Omega)^2\to\mathbb{R}$ and $b(\cdot,\cdot)_\Omega:H^2(\Omega)^2\times L^2(\Omega)\to\mathbb{R}$ defined by
\begin{align*}
	a({\bu},\bv)_\Omega = \int_\Omega \nabla{\bu}:\nabla\bv \du x,\ b(\bv,q)_\Omega = -\int_\Omega q\dive\bv\du x.
\end{align*}
Furthermore, for any measurable set $\mathcal{D}\subset \mathbb{R}^d$ ($d=1,2$) we shall denote $(\cdot,\cdot)_\mathcal{D}$ as the $L^2(\mathcal{D})$, $L^2(\mathcal{D})^2$ or $L^2(\mathcal{D})^{2\times2}$ inner product.

\subsection{Weak formulation and existence of solutions}

Let $\bg\in H^2(D)^2$ satisfy the following properties:
\begin{align}
	\left\{
	\begin{aligned}
	&\Div \bg = 0\text{ in }\Omega;\quad \bg = 0\text{ on }\Gamma_{\rm f}\cup\Gamma_{\rm wall};\\
	&\int_{\Gamma_{\rm out}}\bg\cdot\bn \du s = -\int_{\Gamma_{\rm in}}\bg\cdot\bn \du s \ge 0,
	\end{aligned}
		\right.\label{gprop}
\end{align}
and $\blf\in L^2(D)^2$.  The existence of the function $\bg\in H^2(D)^2$ that satisfies \eqref{gprop}, can be easily established by utilizing for instance \cite[Lemma 2.2]{girault1986}.

By letting $\tilde{\bu} = {\bu} - \bg$, we consider the  variational form of the state equation \eqref{state} given by: For a given domain $\Omega$, find $(\tilde{\bu},p)\in \bH_{\Gamma_0}^1(\Omega)\times L^2(\Omega)$ that satisfies, for any $(\bphi,q)\in  \bH_{\Gamma_0}^1(\Omega)\times L^2(\Omega)$, the following equations
\begin{equation}
\left\{
\begin{aligned}
  \nu a(\tilde{\bu},\bphi)_\Omega + b(\bphi,p)_\Omega & = (\blf,\bphi)_\Omega - \nu(\nabla\bg,\nabla\bphi)_{\Omega}\\
  b(\tilde{\bu},q)_\Omega & = 0.
 \end{aligned}
  \right.
\label{weak}
\end{equation}

Any pair $(\tilde{\bu},p)\in \bH_{\Gamma_0}^1(\Omega)\times L^2(\Omega)$ that solves the variational equation \eqref{weak} is said to be a \textit{weak solution} to the Stokes equation \eqref{state}.  
The existence of the solution $\tilde{\bu}$, is summarized below.
\begin{theorem}
Let $\Omega$ be of class $C^{1,1},$ $\blf\in L^2(D)^2$, and $\bg\in H^2(D)^2$ satisfy \eqref{gprop}. 
The solution $(\tilde{\bu},p)\in \bH_{\Gamma_0}^1(\Omega)\times L^2(\Omega)$ to the variational problem \eqref{weak} exists such that 
\begin{equation}
\|\tilde{\bu}\|_{\bH_{\Gamma_{0}}^1(\Omega)} + \|p\|_{L^2(\Omega)} \le c(\|\blf\|_{L^2(\Omega)^2} +\|\bg\|_{H^{1}(\Omega)^2} ).
\label{energy}
\end{equation}
for some constant \(c>0\). 
\label{th:wp}
\end{theorem}

The proof of the theorem above can be done by utilizing the fact that the operator $a(\cdot,\cdot)_\Omega$ is coercive, $b(\cdot,\cdot)_\Omega$ satisfies the inf-sup condition \cite{girault1986}, and that the right hand side of the first equation of \eqref{weak} can be seen as an action of an element of $\bH^{-1}(\Omega)$.


\begin{remark}
\noindent(i) The energy estimate can be extended to the hold-all domain $D$, i.e.,
\[\|\tilde{\bu}\|_{\bH_{\Gamma_{0}}^1(\Omega)} + \|p\|_{L^2(\Omega)} \le c(\|\blf\|_{L^2(D)^2} +\|\bg\|_{H^{1}(D)^2}) ,\]
where $c>0$ is dependent on $D$ but not on $\Omega$.
\noindent(ii) As expected,  a more regular domain yields a more regular solution. In particular, if $\Omega$ is of class $C^{k,1}$ for $k \ge 0$ then the weak solution to \eqref{weak} satisfies ${\bu}\in \bH_{\Gamma_0}^1(\Omega)\cap\,H^{k+1}(\Omega)^2$, and as a consequence of the Rellich-Kondrachov embedding theorem \cite[Chapter~5, Theorem~6]{evans1998} the solution is in $C(\overline{\Omega})^2$.  Furthermore, if we instead assume that the outer boundary covers a convex polygonal domain while the inner boundary $\Gamma_{\rm f}$ is of class $C^{1,1}$, the same regularity of the solution is obtained when $k=1$.


\end{remark}

\section{Existence of optimal shapes}\label{sec 3}

\subsection{Cone property and the set of admissible domains}

We begin by assuming, without loss of generality, that $\Gamma_{\rm out}\subset \partial D$. We define the set of admissible domains as a subset of the collection of domains inside the hold-all domain $D$ that possesses the cone property. Furthermore, we define the topology on the said admissible set by the convergence of the corresponding characteristic function of each domain in the $L^\infty$-topology. As highlighted by Henrot and Privat in \cite{privat2010} and is rigorously discussed in \cite{zolesio2011} and \cite{henrot2018}, this approach helps in the preservation of volume which is among the goals in our exposition. Several authors utilized parametrizing the free-boundary (c.f. \cite{rabago2019,rabago2020} and the references therein) to define the set of admissible domains, however free-boundary parametrization might lead to generating domains with varying volumes.

We shall adapt the definition of the {\it cone property} as in \cite{chenais1975}.  In what follows, $(\cdot,\cdot)$ and $\|\cdot\|$ denote the inner product and the norm in $\mathbb{R}^2$, respectively.

\begin{definition} Let $h>0$,  $2h>r>0$, $\theta\in[0,2\pi]$, and \(\xi\in\mathbb{R}^2\) such that $\|\xi\| = 1$. 

(i) The cone of angle $\theta$, height $h$ and axis $\xi$ is defined by
\[ C(\xi,\theta,h) = \{x\in\mathbb{R}^2; (x,\xi)>\|x\|\cos\theta, \|x\|<h \}.\]

(ii) A set $\Omega\subset\mathbb{R}^2$ is said to satisfy the \text{cone property} if and only if for all $x\in\partial\Omega$, there exists $C_x=C(\xi_x,\theta,h)$, such that for all $y\in B(x,r)\cap\Omega$ we have $y+C_x \subset\Omega$.
\end{definition}

From this definition,  the set of admissible domains  as 
\[ \Oad := \{\Omega\subset D; \Omega\text{ satisfies the }cone\ property\text{ and }|\Omega|=m \}. \]
This set of admissible domains has been established to be non-empty (see the proof of Proposition 4.1.1 in \cite{henrot2018}), which exempts us from the futility of the analyses we will be discussing.

A sequence $\{\Omega_n\}_n\subset\Oad$ is said to converge to $\Omega\in\Oad$ if 
\[\chi_{\Omega_n}\ws\chi_{\Omega} \text{ in }L^\infty(D),\]
where the function $\chi_A$ for a set $A\subset\mathbb{R}^2$ refers to the characteristic function defined by
\[\chi_A(x) =\left\{\begin{matrix} 1&\text{if } x\in A,\\ 0&\text{if }x\not\in A. \end{matrix}\right.   \] 

\begin{remark}
As explained by Henrot, A., et.al., \cite[Proposition 2.2.1]{henrot2018},  the weak$^*$ convergence in $L^\infty$ implies that the convergence also holds in the space $L^p_{loc}(\mathbb{R}^2)$ and thus $\chi_\Omega$ is almost everywhere a characteristic function.
\end{remark}

We shall denote the collection of characteristic functions of elements of $\Oad$ as $\mathcal{U}_{ad}$, i.e., $\mathcal{U}_{ad}=\{\chi_\Omega; \Omega\in\Oad \}$, and whenever we mention $\mathcal{U}_{ad}$ we take into account the weak$^*$ topology in $L^\infty(D)$.
We refer the reader to \cite[Chapter 5]{zolesio2011} and \cite[Chapter 2~Section 3]{henrot2018} for a more detailed discussion on the topology of characteristic functions of finite measurable domains.

The compactness of $\Oad$ follows from the fact that it is closed and relatively compact -- as defined by Chenais, D. \cite{chenais1975} -- with respect to the topology on $\mathcal{U}_{ad}$. One can also read upon the proof in \cite[Proposition 2.4.10]{henrot2018}. We shall not discuss the proof of such properties, nevertheless they are summarized on the lemma below.
\begin{lemma}
The set $\Oad$ is compact with respect to the topology on $\mathcal{U}_{ad}$.
\label{le:comome}
\end{lemma}

Another important implication of the cone property is the existence of a uniform extension operator.
\begin{lemma}[cf. \cite{chenais1975}]
Let $d=\{1,2 \}$ and $m\in\mathbb{N}\cup\{0\}$, there exists $K>0$ such that for all $\Omega\in\Oad$, there exists \[\mathcal{E}^d_{\Omega}:H^{m}(\Omega)^d\to H^{m}(D)^d\] which is linear and continuous such that  $\displaystyle\max\{\|\mathcal{E}^d_\Omega\|\}_{d = 1,2}\le K$.
\label{uniext}
\end{lemma}
These result will be instrumental for proving some vital properties, for example when we show that the \text{domain-to-state} map is continuous.

\begin{remark} The set of admissible domains can be identified as a collection of Lipschitzian domains. Since we only consider the hold-all domain $D$ to possess a bounded boundary we refer the reader to \cite[Theorem 2.4.7]{henrot2018} for the proof of such property.
\label{lips}
\end{remark}

\subsection{Well-posedness of the Optimization Problem}

Before showing that the optimal shape indeed exists, we take note that from Theorem \ref{th:wp} the following map is well-posed:
\[
\Omega\in \Oad \mapsto (\tilde{\bu}(\Omega),p(\Omega))\in \bH_{\Gamma_0}^1(\Omega)\times L^2(\Omega).
\]
This implies that we can write the objective functional in the manner that it solely depends on the domain $\Omega$, i.e., 
\[ 
	\mathcal{J}(\Omega) := \mathcal{G}({\bu}(\Omega),\Omega),
\]
where ${\bu}(\Omega) = \tilde{{\bu}}+\bg.$ However, the mentioned well-posedness of the domain-to-state map is insufficient to prove the existence of the minimizing domain. In fact, we need the continuity of the map $\Omega\mapsto(\tilde{\bu}(\Omega),p(\Omega))$.

\begin{proposition}
Let $\{\Omega_n\}_n\subset\Oad$ be a sequence that converges to $\Omega\in\Oad$. Suppose that for each $\Omega_n$, $(\tilde{\bu}_n,p_n)\in\bH_{\Gamma_0}^1(\Omega_n)\times L^2(\Omega_n)$ is the weak solution of the Stokes equations on the respective domain; then the extensions $(\overline{{\bu}}_n,\overline{p}_n):=(\mathcal{E}^2_{\Omega_n}\tilde\bu_n,\mathcal{E}^1_{\Omega_n}p_n)\in \bH_{\Gamma_0}^1(D)\times L^2(D)$ coverges to a state $(\overline{\bu},\overline{p})\in \bH_{\Gamma_0}^1(D)\times L^2(D)$, such that $(\tilde{\bu},p)=(\overline{\bu},\overline{p})\big|_{\Omega}\in\bH_{\Gamma_0}^1(\Omega)\times L^2(\Omega)$ is a solution to \eqref{weak}.
\label{contu}
\end{proposition}
\begin{proof}
From the uniform extension property, there exist $K>0$ such that 
\[\|\overline{\bu}_n\|_{\bH_{\Gamma_0}^1(D)} + \|\overline{p}\|_{L^2(D)} \le K(\|\tilde\bu_n\|_{\bH_{\Gamma_0}^1(\Omega_n)} + \|p_n\|_{L^2(\Omega_n)})\text{ for all }\Omega_n.\] Furthermore, from Theorem \ref{th:wp}
\[\|\tilde\bu_n\|_{\bH_{\Gamma_{0}}^1(\Omega_n)} + \|p_n\|_{L^2(\Omega_n)} \le c(\|\blf\|_{L^2(D)^2} +\|\bg\|_{H^{1}(D)^2}) .\] 
This implies uniform boundedness of $\{(\overline{\bu}_n,\overline{p}_n)\}_n$ in $\bH_{\Gamma_0}^1(D)\times L^2(D)$. Hence, by Rellich--Kondrachov's theorem, there exists a subsequence of $\{(\overline{\bu}_n,\overline{p}_n)\}_n$, which shall also be denoted as $\{(\overline{\bu}_n,\overline{p}_n)\}_n$, and an element $(\overline{\bu},\overline{p})\in\bH_{\Gamma_0}^1(D)\times L^2(D)$ such that the following properties hold:
\begin{align}
\left\{
\begin{aligned}
	&\overline{\bu}_n \rightharpoonup \overline{\bu} &&\text{in }H^1(D)^2,\\
	&\overline{\bu}_n \to \overline{\bu}	&&\text{in }L^2(D)^2,\\
	&\overline{p}_n  \rightharpoonup \overline{p}		&&\text{in }L^2(D).
\end{aligned}\right.\label{solcon}
\end{align} 

\noindent\textit{Passing through the limit.}
We shall now show that the limit $(\overline{\bu},\overline{p})$, when restricted to the domain $\Omega$, solves \eqref{weak}. 
Indeed, since $(\tilde{\bu}_n,{p}_n) = (\overline{\bu}_n,\overline{p}_n)\big|_{\Omega_n}$ solves \eqref{weak} in $\Omega_n$, then for any $(\bphi,q)\in \bH^1_{\Gamma_{0}}(D)\times L^2(D)$
\begin{equation}
	\left\{
	\begin{aligned}
	\nu(\chi_n\nabla\overline\bu_n,\nabla\bphi)_D + b(\bphi,\chi_n\overline{p}_n)_D & = (\chi_n\blf,\bphi)_D -  \nu(\chi_n\nabla\bg,\nabla\bphi)_{D},\\
	b(\overline\bu_n,\chi_n q)_D & = 0.
	\end{aligned}
	\right.
\label{weakDn}
\end{equation}
Furthermore, since $\chi_n:=\chi_{\Omega_n}\ws \chi:=\chi_\Omega$ in $\mathcal{U}_{ad}$ and from \eqref{solcon}, we can easily show that 
\begin{equation}
	\left\{
	\begin{aligned}
	\nu(\chi\nabla\overline\bu,\nabla\bphi)_D + b(\bphi,\chi\overline{p})_D & = (\chi\blf,\bphi)_D - \nu(\chi\nabla\bg,\nabla\bphi)_{D},\\
	b(\overline\bu,\chi q)_D & = 0.
	\end{aligned}
	\right.
\label{weakD}
\end{equation}
Thus, $(\tilde\bu,p)=(\overline{\bu},\overline{p})\big|_{\Omega}$ is a solution to \eqref{weak}. 
\hfill
\end{proof}

As usual in proving the existence of solution for minimization problems, the arguments will be based on the lower semicontinuity, and the existence of a minimizing sequence based on the boundedness below of the objective functional.

Note that the objective functional $\mathcal{G}$ can be dissected into three different integrals, namely $J_1(\bu,\Omega) = \frac{\gamma_1}{2}\|\nabla\times\bu\|_{L^2(\Omega)}^2$, $J_2({\bu},\Omega) = {\gamma_2}\|h(\mathrm{det}(\nabla{\bu}))\|_{L^1(\Omega)}$, and $P(\Omega)= \int_{\Gamma_{\rm f}}\du s$. Thus, to establish that $\mathcal G$ is bounded below, we can show that $J_1$ and $J_2$ are uniformly bounded. This implies that, since the boundary $\Gamma_{\rm f}$ is strictly inside the bounded hold-all domain $D$, $C \le \mathcal G$ for some constant $C$.

Note that $J_1$ can be estimated by the $H^{1}(\Omega)^2$ norm of the state ${\bu}$, i.e.,
\begin{align*}
J_1({\bu},\Omega) & = \frac{\gamma_1}{2}\|\nabla\times{\bu}\|_{L^2(\Omega)}^2  \le  \frac{\gamma_1}{2}\int_{\Omega} |\nabla{\bu}|^2\du x \le  \frac{\gamma_1}{2}\|{\bu}\|^2_{H^{1}(\Omega)^2}.
\end{align*}
Meanwhile, since $t^3/(t^2+1)\le t$ for any $t$, we get the following estimate,
 \[J_2({\bu},\Omega)\le \gamma_2\|\mathrm{det}(\nabla{\bu})\|_{L^1(\Omega)}\le \frac{1}{2}\|{\bu}\|_{H^{1}(\Omega)^2}^2,\]
where Young's inequality has been employed for the last inequality.

From Proposition \ref{contu}, we have established the continuity of $\Omega\mapsto \tilde{\bu}$, which also implies the continuity of the map $\Omega \mapsto {\bu} = \tilde{\bu}+\bg\in H^{1}(\Omega)^2$. Thus, from the recently established estimates for $J_1$ and $J_2$ with respect to ${\bu}$, we can infer that the respective functionals are continuous (and hence upper semicontinuous) with respect to the elements of $\Oad\,$.
Furthermore, the functionals are uniformly bounded, since for each $i=1,2$,
\begin{align*}
	J_i({\bu},\Omega) & \le c\|{\bu}\|_{H^{1}(\Omega)^2}^2 = c\|\tilde\bu+\bg\|_{H^{1}(\Omega)^2}^2  \le \tilde{c}(\|\blf\|_{L^2(D)^2} +\|\bg\|_{H^{1}(D)^2} )^2.
\end{align*}

Given all these facts, we can now show the existence of an optimal shape rendering our problem well-posed. 
\begin{theorem}
Let $\alpha,\gamma_1,\gamma_2 > 0$, ${\blf}\in L^2(D)^2$, and $\bg\in H^2(D)^2$ satisfy \eqref{gprop}. Then, there exists $\Omega^*\in\Oad$ such that 
\[ \mathcal{J}(\Omega^*) = \min_{\Omega\in\Oad}\mathcal{J}(\Omega).\]
\end{theorem}
\begin{proof}
From the fact that $\mathcal{J}$ is bounded below, there exists a sequence $\{\Omega_n\}\subset\Oad$ such that 
\[ \mathcal{J}^* : = \inf_{\Omega\in\Oad}\mathcal{J}(\Omega) = \lim_{n\to\infty}\mathcal{J}(\Omega_n).\]

As a consequence of Lemma \ref{le:comome}, there exists a subsequence of $\{\Omega_n\}$, which shall be denoted in the same manner, and a domain $\Omega^*$ such that $\Omega_n\to\Omega^*$ in $\Oad\, $.
By definition, $\mathcal{J}^*\le \mathcal{J}(\Omega^*)$. However, since $-J_1$, $-J_2$, and $P$ (for $P$, see Proposition 2.3.7 in \cite{henrot2018}) are lower semicontinuous,
\begin{align*}
\mathcal{J}(\Omega^*) & = P(\Omega^*) - (J_1(\Omega^*) + J_2(\Omega^*)) \le \liminf_{n\to\infty} [P(\Omega_n) - (J_1(\Omega_n) + J_2(\Omega_n))]\\ & = \liminf_{n\to\infty} \mathcal{J}(\Omega_n) = \mathcal{J}^*.
\end{align*}
Therefore, $\Omega^*$ is our desired minimizer.\hfill
\end{proof}

\section{Shape Sensitivity Analysis}\label{sec 4} 
In this section, given the previously established existence of the optimal shape, we shall discuss the necessary condition for the optimization problem. This is done by investigating the sensitivity of the objective functional $\mathcal{G}$ with respect to shape variations. We start this section by introducing the velocity method, where we consider domain variations generated by a given velocity.

\subsection{Identity perturbation}
In what follows, we consider a family of autonomous deformation fields ${\bth}$ belonging to ${\bTh} := \{{\bth}\in C^{1,1}(\overline{D};\mathbb{R}^2); {\bth} = 0 \text{ on }\partial D\cup\Gamma_{\rm in}\cup\Gamma_{\rm w}\}$. An element ${\bth}\in\bTh$ generates an identity perturbation operator $T_t:\overline{D}\to\mathbb{R}^2$ defined by
\begin{align}
\begin{aligned}
T_t(x) = x + t{\bth}(x), \quad\forall x \in \overline{D}.
\end{aligned}
\label{transform}
\end{align}
With this operator, a domain $\Omega\subset D$ is perturbed so that for some $\tau:=\tau({\bth})>0$ we have a family of perturbed domains $\{\Omega_t = T_t(\Omega); t<\tau \}$. Here, $\tau$ is chosen so that $\mathrm{det}\nabla T_t>0$, where the following statement shows that $\tau$ is indeed dependent to the velocity field ${\bth}$.
\begin{lemma}[cf. \cite{bacani2013,haslinger2006}]
Let ${\bth}\in\bTh$ and $T_t$ be the generated transformation by means of \eqref{transform} and denote $J_t := \mathrm{det}\nabla T_t$. Then, we have the following:
\begin{itemize}
	\item[(i)] $J_t = 1 + t\dive{\bth} + t^2\mathrm{det}\nabla{\bth}$; 
	\item[(ii)] there exists $\tau_0=\tau_0({\bth}),\, \alpha_1,\alpha_2>0$ independent of $t$ and $x$ such that 
	\[ 0<\alpha_1\le J_t(x)\le \alpha_2,\quad\forall t\in[0,\tau_0],\forall x\in D. \]
\end{itemize}
\label{detbound}
\end{lemma}
We further mention that by the definition of ${\bth}$, i.e., ${\bth}\equiv 0$ on $\Gamma_{\rm out}$, $\Gamma_{\rm in}$, and $\Gamma_{\rm w}$, then these boundaries are part of the perturbed domains $\Omega_t$. To be precise, we have
\[ \partial\Omega_t = \Gamma_{\rm out}\cup\Gamma_{\rm in}\cup\Gamma_{\rm w}\cup T_t(\Gamma_{\rm f}). \]
Additionally, a domain that has at most $C^{1,1}$ regularity preserves its said regularity with this transformation, this means that for $0\le t\le \tau$, $\Omega_t$ has $C^{1,1}$ regularity given that the initial domain $\Omega$ is a $C^{1,1}$ domain. Lastly, we note that $\{\Omega_t = T_t(\Omega); t<\tau \}\subset \Oad$ due to Remark \ref{lips}.

Before we move further in this exposition, let us look at some vital properties of $T_t$. 
\begin{lemma}[cf \cite{zolesio2011,sokolowski1992}]
Let ${\bth}\in\bTh$, then for sufficiently small $\tau>0$, the map $T_t$ defined in \eqref{transform} satisfies the following properties:
\begin{itemize}
	\item[$\bullet$] $[t\mapsto T_t]\in C^1([0,t_0];C^{2,1}(\overline{D},\mathbb{R}^2));\ \quad\hspace{-.15in} \bullet\ [t\mapsto T_t^{-1}]\in C([0,t_0];C^{2,1}(\overline{D},\mathbb{R}^2));$
	\item[$\bullet$] $[t\mapsto J_t]\in C^1([0,t_0];C^{1,1}(\overline{D}));\quad\qquad\hspace{-.15in}\! \bullet\ M_t,M_t^\top\in C^{1,1}(\overline{D},\mathbb{R}^{2\times2});$
	\item[$\bullet$] $\frac{d}{dt}J_t\big|_{t=0} = \dive\bth;\hspace{-.15in}\quad\qquad\qquad\qquad\quad\ \bullet\ \frac{d}{dt}M_t\big|_{t=0} = -\nabla\bth,$
\end{itemize}
where $M_t(x) = (\nabla T_t(x))^{-1}$.
\label{Tprops}
\end{lemma}

Let us also recall Hadamard's identity which will be indespensible for the discussion of the necessary conditions.
\begin{lemma}
Let $f\in C([0,\tau];W^{1,1}(D))$ and suppose that $\frac{\partial f}{\partial t}(0)\in L^1(D)$, then 
\begin{align*}
	\frac{d}{dt}\int_{\Omega_t}f(t,x)\du x\Big|_{t = 0} = \int_\Omega \frac{\partial f}{\partial t}(0,x)\du x + \int_{\Gamma_{\rm f}} f(0,x){\bth}\cdot\bn\du s.
\end{align*}
\label{hadamard}
\end{lemma}
\begin{proof}
See Theorem 5.2.2 and Proposition 5.4.4 of \cite{henrot2018}.\hfill
\end{proof}
\subsection{Rearrangement Method} 
In this part, we determine the shape derivative of the objective functional with respect to the variations generated by the transformation $T_t$. This approach gets rid of the tedious process of solving first the {\it shape derivative} of the state solutions, then solving the {\it shape derivative} of the objective functional.

To start with,  we consider a Hilbert space $Y(\Omega)$ and an operator
\[F:Y(\Omega)\times \Oad\to Y(\Omega)',\]
where the equation $\langle F(y,\Omega), \phi \rangle_{Y(\Omega)'\times Y(\Omega)} = 0$ corresponds to a variational problem in $\Omega.$

Suppose that the free-boundary is denoted by $\Gamma_{\rm f}\subset \partial\Omega$, the said method deals with the shape optimization 
\[ \min_{\Omega\in \Oad}J(y,\Omega):=\int_\Omega j(y)dx +\alpha \int_{\Gamma_{\rm f}}\du s. \]
subject to 
\begin{align}
 F(y,\Omega)=0\text{ in }Y(\Omega)'.\label{weakgen}
\end{align}
We define the Eulerian derivative of $J$ at $\Omega$ in the direction ${\bth}\in\bTh$ by
\[ dJ(y,\Omega){\bth} = \lim_{t\searrow0}\frac{J(y_t,\Omega_t) - J(y,\Omega)}{t},\]
where $y_t$ solves the equation $F(y_t,\Omega_t)=0$ in $Y(\Omega_t)'$. If $dJ(y,\Omega){\bth}$ exists for all ${\bth}\in\bTh$ and that $dJ(y,\Omega)$ defines a bounded linear functional on $\bTh$ then we say that $J$ is {\it shape differentiable} at $\Omega$.

The so-called rearrangement method is given as below:
\begin{lemma}[cf \cite{ito2008}]
Suppose $j\in C^{1,1}(\mathbb{R}^2,\mathbb{R})$ and that the following assumptions hold:
	\begin{pethau}
	\item[(A1)] There exists an operator $\tilde{F}:Y(\Omega)\times[0,\tau]\to Y(\Omega)'$ such that $F(y_t,\Omega_t)=0$ in $Y(\Omega_t)'$ is equivalent to 
	\begin{align}
		\tilde{F}(y^t,t) = 0\text{ in }Y(\Omega)',\label{weakbacktrack}
	\end{align}
	 with $\tilde{F}(y,0) = F(y,\Omega)$ for all $y\in Y(\Omega).$
	\item[(A2)]  Let $y,v\in Y(\Omega)$. Then $F_y(y,\Omega)\in \mathcal{L}(Y(\Omega),Y(\Omega)')$ satisfies
		\[ \langle F(v,\Omega)-F(y,\Omega)-F_y(y,\Omega)(v-y),z\rangle_{Y(\Omega)'\times Y(\Omega)} = \mathcal{O}(\|v-y \|_{Y(\Omega)}^2),\]
		for all $z\in Y(\Omega). $
	\item[(A3)] Let  $y\in Y(\Omega)$ be the unique solution of \eqref{weakgen}. Then for any $f\in Y(\Omega)'$  the solution of the following linearized equation exists:
	\begin{align*}
		\langle F_y(y,\Omega)\delta y,  z\rangle_{Y(\Omega)\times Y(\Omega)'} = \langle f,  z\rangle_{Y(\Omega)\times Y(\Omega)'} \text{ for all }z\in Y(\Omega).
	\end{align*}
	\item[(A4)] Let $y^t,y\in Y(\Omega)$ be the solutions of \eqref{weakbacktrack} and \eqref{weakgen}, respectively. Then $\tilde F$ and $F$ satisfy
	\[\lim_{t\searrow 0}\frac{1}{t}\langle \tilde F(y^t,t) - \tilde{F}(y,t) -F(y^t,\Omega) + F(y,\Omega) ,z\rangle_{Y(\Omega)'\times Y(\Omega)} = 0\]
	for all $z\in Y(\Omega). $
\end{pethau}
Let $y\in Y(\Omega)$ be the solution of \eqref{weakgen}, and suppose that the adjoint equation,  for all $z\in Y(\Omega)$
\begin{align}
\langle F_y(y,\Omega)z,p\rangle_{Y(\Omega)'\times Y(\Omega)} = (j'(y),z)_{L^2(\Omega)} 
\label{weakadjoint}
\end{align}
has a unique solution $p\in Y(\Omega)$.  Then the Eulerian derivative of $J$ at $\Omega$ in the direction ${\bth}\in\bTh$ exists and is given by
\begin{align}
\begin{aligned}
dJ(y,\Omega){\bth} =&\, -\frac{d}{dt}\langle \tilde{F}(y,t), p\rangle_{Y(\Omega)'\times Y(\Omega)}\Big|_{t=0}\\ &\, + \int_\Omega j(y)\dive{\bth} \du x + \alpha\int_{\Gamma_{\rm f}}\kappa {\bth}\cdot\bn \du s,
\end{aligned}
\label{necessaryconditiongen}
\end{align}
where $\kappa$ is the mean curvature of $\Gamma_{\rm f}$.
\label{rearrangement}
\end{lemma}

Before we go further, we note that in usual practice the element $y^t \in Y(\Omega)$ is solved using the composition $y^t : = y_t\circ T_t$, and utilizes the function space parametrization (see \cite[Chapter 10 Section 2.2]{zolesio2011} for example, for further details). Furthermore, assumption {\it (A3)} is originally written as the H{\"o}lder continuity of solutions $y^t\in Y(\Omega)$ of \eqref{weakbacktrack} with respect to the time parameter $t\in[0,\tau]$.  Fortunately, Ito, K., et.al.\cite{ito2008}, have shown that assumption {\it (A3)} implies the continuity. We cite the said result in the following lemma.

\begin{lemma}
Suppose that $y\in Y(\Omega)$ solves \eqref{weakgen} and $y^t\in Y(\Omega)$ is the solution to \eqref{weakbacktrack}. Assume furthermore that {\it (A3)} holds. Then, $\|y^t-y \|_{X(\Omega)} = o(t^\frac{1}{2})$ as $t\searrow 0$.
\label{holder}
\end{lemma}

We applying the rearrangement method to the velocity-pressure operator $E(\cdot)_\Omega:X(\Omega)\to X(\Omega)' $ defined by
\begin{align*}
\langle E({\bu},p)_\Omega,(\bphi, \psi) \rangle_{X(\Omega)'\times X(\Omega)} = &\,  \nu a(\bu,\bphi)_\Omega + b(\bphi,p)_\Omega + b({\bu},q)_\Omega \\ &- (\blf,\bphi)_\Omega + a(\bg,\bphi)_\Omega
\end{align*} 
where $X(\Omega) := \bH^1_{\Gamma_0}(\Omega)\times L^2(\Omega)$.

According to Theorem \ref{th:wp},  there exists a unique $(\tilde\bu, p)\in X(\Omega)$ such that for any $(\bphi, \psi)\in X(\Omega)$ 
\begin{align}
 \langle E(\tilde\bu,p)_\Omega,(\bphi, \psi) \rangle_{X(\Omega)'\times X(\Omega)} = 0.\label{operator}
 \end{align}

\begin{notation} {Moving forward we shall use the following notations $X = X(\Omega)$,  $X_t = X(\Omega_t) $.}\end{notation}

Our goal is to characterize, and of course show the existence of the Eulerian derivative of the objective functional 
\begin{align}
\mathcal{G}({\bu},\Omega) =  \alpha\int_{\Gamma_{\rm f}} \du s-\int_{\Omega} \frac{\gamma_1}{2}|\nabla\times {\bu}|^2 + \gamma_2 g(\mathrm{det}(\nabla {\bu})) \du x,
\end{align}
where ${\bu} = \tilde\bu + \bg$, and $\tilde{\bu}\in \bH_{\Gamma_0}^1(\Omega)$ is the first component of the solution of \eqref{operator}.

From the deformation field $T_t$, we let $(\tilde\bu_t,p_t)\in X_t$ be the solution of the equation
\begin{align} 
\langle E(\tilde\bu_t,p_t)_{\Omega_t},(\bphi_t, \psi_t) \rangle_{X_t'\times X_t} = 0, \quad\forall(\bphi_t, \psi_t)\in X_t.
\label{operatoront}
\end{align}
By perturbing the equation \eqref{operatoront} back to the reference domain $\Omega$ we get the operator $\tilde{E}:X\times[0,\tau]\to X'$ defined by
\begin{align}
\begin{aligned}
\langle\tilde{E}((\bu,p),t),(\bphi,\psi) \rangle_{X'\times X}  &=\nu a_t({\bu},\bphi) + b_t(\bphi,p)  + b_t({\bu},\psi) \\ &-  (J_t\blf^t,\bphi)_{\Omega}+\nu a_t({\bg},\bphi)
\end{aligned}
\end{align}
where $\blf^t = \blf\circ T_t\in L^2(\Omega)^2 $ and -- by denoting $(M_t^\top)_k$ the $k^{th}$ row of $M_t^\top$,  $\bv_k$ the $k^{th}$ component of a vector $\bv$, and by using Einstein convention over $k=1,2$ -- the bilinear forms are defined as follows
\begin{align*}
 a_t({\bu},\bphi) =&\, (J_tM_t\nabla{\bu},M_t\nabla\bphi)_\Omega,\\
b_t(\bv,\psi) = &\, -(J_t \psi, (M_t^\top)_k\nabla\bv_k)_\Omega.
\end{align*}
Using similar arguments as in \cite{ito2008}, it can be easily shown that $E(\cdot,\cdot)_\Omega$ and $\tilde{E}$ satisfy (A1) and (A4).

Furthermore, the Fr{\'e}chet derivative $E'\in\mathcal{L}(X,\mathcal{L}(X,X'))$ of the velocity-pressure operator at a point $({\bu},p)\in X$ can be easily obtained as 
\[ 
	\langle E'({\bu},p)(\delta{\bu},\delta p), (\bphi,\psi) \rangle_{X'\times X} = \nu a(\delta{\bu},\bphi)_\Omega + b(\bphi,\delta p)_\Omega + b(\delta{\bu},q)_\Omega,
\]
for all $(\bphi,\psi)\in X$.  Due to the linearity, we further infer that 
\[
	\langle E({\bv},q)_\Omega -E({\bu},p)_\Omega-  E'({\bu},p)({\bv}-{\bu},q-p), (\bphi,\psi) \rangle_{X'\times X} = 0
\]
for any $({\bv},q),({\bu},p),(\bphi,\psi)\in X$.  From the coercivity of $a(\cdot,\cdot)_{\Omega}$, and because $b(\cdot,\cdot)_\Omega$ satisfies the inf-sup condition,  then for any $\Phi\in X'$ the variational problem
\[ \langle E'(\tilde{\bu},p)(\delta{\bu},\delta p), (\bphi,\psi) \rangle_{X'\times X} = \langle \Phi, (\bphi,\psi) \rangle_{X'\times X},\quad\forall(\bphi,\psi)\in X,\]
is well-posed. These imply that the operators $E(\cdot,\cdot)_\Omega$ and $E'({\bu},p)$ satisfy assumptions { (A2) and (A3)}. 


We now characterize the Eulerian derivative of $\mathcal{G}$ by utilizing Lemma \ref{rearrangement}.
\begin{theorem}
	Suppose that $\Omega\subset D$ is a domain that is of class $C^{1,1}$,  and let $(\tilde\bu,p)\in (\bH_{\Gamma_{0}}^1(\Omega)\cap H^2(\Omega)^2)\times L^2(\Omega)$ be the solution to \eqref{operator}.  Then, the Eulerian derivative of $\mathcal{G}$ exists and is given by
	\begin{align*}
		d\mathcal{G}((\bu,p),\Omega){\bth} = &\,\int_{\Gamma_{\rm f}}\bigg[ \alpha \kappa - \frac{\gamma_1}{2}|\nabla\times{\bu}|^2 - \gamma_2 h(\mathrm{det}(\nabla{\bu}))\\
		& + \frac{\partial{\bu}}{\partial{\bn}}\left(\frac{\partial{\bv}}{\partial{\bn}} + \gamma_1(\nabla\times{\bu}){\btau} + \gamma_2P({\bu}) \right)  \bigg]{\bth}\cdot{\bn}\du s,
	\end{align*}
	where $\btau$ is the unit tangential vector on $\Gamma_{\rm f}$, ${\bu} = \tilde{\bu}+{\bg}\in H^2(\Omega)^2$, $(\bv,\pi)\in X$ solves the adjoint equation
	\begin{align}
		\langle E'(\tilde{\bu},p)({\bphi},\psi),({\bv},\pi) \rangle_{X'\times X} =  -(\gamma_1\vec{\nabla}\times(\nabla\times{\bu}) + \gamma_2R(\bu),{\bphi})_{\Omega}, 
	\label{wkadjoint}
	\end{align}
	for any $(\bphi,\psi)\in X$, 
		\[ 
		R(\bu) = 	\begin{bmatrix}
							-\nabla\times(h'(\mathrm{det}(\nabla{\bu}))\nabla u_2 )\\
							\nabla\times(h'(\mathrm{det}(\nabla{\bu}))\nabla u_1 )
						\end{bmatrix}		\!, P(\bu) = 	\begin{bmatrix}
							h'(\mathrm{det}(\nabla{\bu}))\left(\frac{\partial u_2}{\partial x_2}n_{x_1} - \frac{\partial u_2}{\partial x_1}n_{x_2} \right) \\
							h'(\mathrm{det}(\nabla{\bu}))\left(\frac{\partial u_1}{\partial x_1}n_{x_2} - \frac{\partial u_1}{\partial x_2}n_{x_1} \right) 
						\end{bmatrix}		
	\]
	 and the curl of a scalar valued function is defined by $\vec{\nabla}\times \psi = \left(\frac{\partial \psi}{\partial x_2}, -\frac{\partial \psi}{\partial x_1}\right)^\top.$
	\label{th:neccond}
\end{theorem}

\begin{remark}
	In strong form, we can write the adjoint equation \eqref{wkadjoint} as follows:

		\begin{align}
	\left\{
	\begin{aligned}
	\, - \nu\Delta{\bv} +  \nabla\pi
	&= -[\gamma_1\vec{\nabla}\times(\nabla\times{\bu})+\gamma_2R({\bu})] && \text{ in } \Omega,\\	
	\, \Div{\bv} &= 0 && \text{ in } \Omega,\\ 		
	\, {\bv} &= 0 && \text{ on }\partial\Omega\backslash\Gamma_{\rm out},\\
	\, -\pi{\bn} + \nu \partial_{\bn}{\bv} & = 0&&\text{ on }\Gamma_{\rm out}.
	\end{aligned} \right.
	\label{adjoint}
\end{align}
\end{remark}

\noindent{\it Proof of Theorem \ref{th:neccond} } From the symmetry of the operator $E'(\tilde{\bu},p)\in \mathcal{L}(X',X)$, the unique existence of the solution to the adjoint equation \eqref{wkadjoint} can be established easily using the same arguments as in \eqref{weak} -- this of course uses the fact that $\bu\in H^2(\Omega)^2$, and hence the well-definedness of the right-hand side of \eqref{wkadjoint}.  Thus, the existence of the Eulerian derivative of $\mathcal{G}$ is assured.

We now characterize the derivative $d\mathcal{G}(({\bu},p),\Omega){\bth}$. We begin by solving for the expression $\frac{d}{dt}\langle \tilde{E}((\tilde{\bu},p),t), ({\bv},\pi)\rangle_{X'\times X}\big|_{t=0}$.
By pushing $\tilde{E}(({\bu},p),t)$ toward the perturbed domain $\Omega_t$ and by using Lemma \ref{hadamard} yield the following
\begin{align*}
	\frac{d}{dt}\langle \tilde{E}((\tilde{\bu},p),t), &({\bv},\pi)\rangle_{X'\times X}\Big|_{t=0}\\
	= &\, [\nu(\nabla{\bu}:\nabla{\bv},{\bth}\cdot{\bn})_{\Gamma_{\rm f}} - (p\dive{\bv},{\bth}\cdot{\bn})_{\Gamma_{\rm f}} - (\psi\dive{\bu},{\bth}\cdot{\bn})_{\Gamma_{\rm f}}] \\
	& + [\nu a({\bu},\phi_v)_\Omega + b(\phi_v,p)_\Omega + b({\bu},\phi_{\pi})_\Omega - (\blf\phi_v)_\Omega]\\
	& + [\nu a(\phi_u,{\bv})_\Omega + b(\phi_u,\pi)_\Omega + b({\bv},\phi_q)_\Omega]
	\end{align*}
where $\phi_u = -\nabla{\bu}^\top{\bth}$, $\phi_v = -\nabla{\bv}^\top{\bth}$,  $\phi_p = -\nabla p\cdot{\bth}$, and $\phi_\pi = -\nabla\pi\cdot{\bth}$, which were generated by using the identity $\frac{d}{dt}(\varphi\circ T_t^{-1}) = -\nabla\varphi\cdot{\bth},$ for $\phi\in H^1(D)$ \cite{sokolowski1992}.  
Using Green's identities, the divergence-free property of the variables ${\bu}$ and ${\bv}$ and Green's curl identity \cite[Theorem 3.29]{monk2003}, we obtain the following derivative
\begin{align*}
	\frac{d}{dt}\langle \tilde{E}&((\tilde{\bu},p),t), ({\bv},\pi)\rangle_{X'\times X}\Big|_{t=0}\\
					=&\, -\left(\nu\partial_{\bn}{\bu} \cdot\partial_{\bn}{\bv} ,{\bth}\cdot{\bn}\right)_{\Gamma_{\rm f}} + \gamma_1 \Big[(j_1'({\bu}),\nabla{\bu}^\top{\bth})_{\Omega} - ((\nabla\times{\bu}){\btau}\cdot\partial_{\bn}{\bu},{\bth}\cdot{\bn})_{\Gamma_{\rm f}}\Big]\\
					&+\gamma_2\Big[ (j_2'({\bu}),\nabla{\bu}^\top{\bth})_{\Omega} - (P({\bu})\cdot\partial_{\bn}{\bu},{\bth}\cdot{\bn})_{\Gamma_{\rm f}}\Big].
\end{align*}

Finally,  we now use \eqref{necessaryconditiongen} to determine the derivative $d\mathcal{G}((\bu,p),\Omega){\bth}$ to get
\begin{align*}
	d\mathcal{G}((\bu,p)&,\Omega){\bth}  =   \big( \partial_{\bn}{\bu}\cdot(\nu\partial_{\bn}{\bv} +(\nabla\times{\bu}){\btau} + P({\bu})  ), {\bth}\cdot{\bn}\big)_{\Gamma_{\rm f}}\\ 
	& +\alpha(\kappa, {\bth}\cdot\bn)_{\Gamma_{\rm f}} - \frac{\gamma_1}{2}\big[2(j_1'({\bu}),\nabla{\bu}^\top{\bth})_{\Omega}+ ( j_1({\bu}),\dive{\bth} )_{\Omega}\big]\\
	& -  \gamma_2\big[(j_2'({\bu}),\nabla{\bu}^\top{\bth})_{\Omega}+ (j_2({\bu}),\dive{\bth})_\Omega\big].
\end{align*}
Before we write the final form of the shape derivative, we note that
\begin{align*}
	\int_{\Omega}\dive(j_1({\bu}){\bth})\du x = 2(j_1'({\bu}),\nabla{\bu}^\top{\bth})_{\Omega} + ( j_1({\bu}),\dive{\bth} )_{\Omega},
\end{align*}
and 
\begin{align*}
	\int_{\Omega}\dive(j_2({\bu}){\bth})\du x = (j_2'({\bu}),\nabla{\bu}^\top{\bth})_{\Omega} + ( j_2({\bu}),\dive{\bth} )_{\Omega}.
\end{align*}
Hence, we have the following form
\begin{align*}
	d\mathcal{G}((\bu,p),\Omega){\bth}  =&\,  \int_{\Gamma_{\rm f}}\left[\alpha\kappa + \frac{\partial{\bu}}{\partial{\bn}}\cdot\left(\frac{\partial{\bv}}{\partial{\bn}} + \gamma_1 (\nabla\times{\bu}){\btau} + \gamma_2P({\bu})\right)\right]{\bth}\cdot\bn\du s\\
	& - \frac{\gamma_1}{2} \int_{\Omega}\dive(j_1({\bu}){\bth})\du x - \gamma_2\int_{\Omega}\dive(j_2({\bu}){\bth})\du x.
\end{align*}
Therefore, from the assumed regularity of the domain, and from the divergence theorem
\begin{align*}
	d\mathcal{G}((\bu,p),\Omega){\bth}  =&\,  \int_{\Gamma_{\rm f}}\left[\alpha\kappa + \frac{\partial{\bu}}{\partial{\bn}}\cdot\left(\frac{\partial{\bv}}{\partial{\bn}} + \gamma_1(\nabla\times{\bu}){\btau} + \gamma_2P({\bu})\right)\right]{\bth}\cdot\bn\du s\\
	& - \frac{\gamma_1}{2} \int_{\Gamma_{\rm f}}j_1({\bu}){\bth}\cdot{\bn}\du s - \gamma_2\int_{\Gamma_{\rm f}}j_2({\bu}){\bth}\cdot{\bn}\du s.
\end{align*}
\hfill

\begin{remark}
	The resulting form of the shape derivative of $\mathcal{G}$ coincides with the Zolesio--Hadamard {\it Structure Theorem} \cite[Corollary 9.3.1]{zolesio2011}, i.e., we were able to write 
	\[ d\mathcal{G}((\bu,p),\Omega){\bth} = \int_{\Gamma_{\rm f}} \nabla{G}({\bth}\cdot{\bn}) \du s.\]
	In this case, we shall call $\nabla G$ the {\it shape gradient} of $\mathcal{G}$ which will be useful in the numerical implementations that we shall illustrate in the proceeding sections.
	\label{rem:gradient}
\end{remark}

\section{Numerical Realization}\label{sec 5}
In this section we shall use the recently formulated necessary condition to solve the minimization problem numerically.  However, the said condition does not take into account the volume constraint. With this reason, we shall resort to two approaches, one is the {\it augmented Lagrangian} method, which will be based on the implementation of Dapogny, C., et.al.\cite{dapogny2018}, and the other one is by generating divergence-free velocity fields.  

\subsection{Augmented Lagrangian Method}

By writing the shape optimization problem \eqref{shapeopti} as the equality constrained optimization
\begin{align*}
&\min_{\Omega\in \Oad} \mathcal{J}(\Omega)
\text{ subject to } \mathcal{F}(\Omega) := |\Omega|-m=0,
\end{align*}
we consider minimizing the augmented Lagrangian 
\[\mathcal{L}(\Omega,\ell,b) : = \mathcal{J}(\Omega) - \ell \mathcal{F}(\Omega) + \frac{b}{2}(\mathcal{F}(\Omega))^2.\]
In this definition, $\ell$ corresponds to the Lagrange multiplier with respect to the equality constraint, while $b>0$ is a penalty parameter. For a more detailed discussion on augmented Lagrangian methods one can refer to \cite[Section 17.3]{nocedal2006}, and to \cite{dapogny2018} for implementation to shape optimization.

We can solve the shape gradient of $\mathcal{L}$ at a given domain $\Omega\in \Oad$ - which we shall denote as $\nabla L$ - by utilizing the shape gradient of $\mathcal{J}$ and Lemma \ref{hadamard}. This yields
\[ \nabla L = \nabla G - \ell + b(|\Omega| - m), \]
where the term $|\Omega| - m$ is evaluated as $\int_{\Omega}\du x - m.$

\subsection{Gradient Descent and Normal Vector Regularization}

Similarly with Remark \ref{rem:gradient}, one can write the shape derivative of $\mathcal{L}$ as
\[ d\mathcal{L}(\Omega,\ell,b){\bth} = \int_{\Gamma_{\rm f}} \nabla L({\bth}\cdot{\bn})\du s. \]
This gives us an intuitive form of the velocity field as either $\bth = -\nabla G{\bn}$ (or $\bth = -\nabla L{\bn}$ for the augmented Lagrangian), which implies that 
\[ d\mathcal{G}(\Omega,\ell,b){\bth} = -\|\nabla G\|^2_{L^2(\Gamma_{\rm f})} \ (\text{or }d\mathcal{L}(\Omega,\ell,b){\bth} = -\|\nabla L\|^2_{L^2(\Gamma_{\rm f})} )\]
However, such choice of $\bth$ may lead to unwanted behavior of the deformed domain especially for iterative schemes.  To circumvent this issue, several authors utilized extensions of the velocity field in the domain $\Omega$. In \cite{dapogny2018},  the authors considered an elasticity extension with the tangential gradient for solving the field. Rabago et.al. \cite{rabago2019} on the other hand cosidered a Neumann-type extension for solving the mentioned velocity field. In our context we shall consider two approaches based on $H^1$ gradient method \cite{azegami2017}. In particular, for sufficiently small parameter $\gamma>0$, and $\lambda\in\{0,1\}$ we shall use the following $H^1$ gradient method: {\it Solve for $(\bth,\vartheta)\in H^1(\Omega)^2\times L^2(\Omega)$ ($\bth\in H^1(\Omega)^2$ if $\lambda = 0$) that satisfies the variational equation}
\begin{align}
	\left\{
		\begin{aligned}
			\gamma a({\bth},{\bphi})_{\Omega} + ({\bth},{\bphi})_{\Omega} + \lambda b(\vartheta,{\bphi})_{\Omega} &= -(\nabla K{\bn},{\bphi})_{\Gamma_{\rm f}}\\
			\lambda b(\psi,{\bth})_{\Omega} & = 0
		\end{aligned}
	\right.
	\label{velsolve2}
\end{align}
{\it for all $(\bphi,\psi)\in H^1(\Omega)^2\times L^2(\Omega)$ ($\bphi\in H^1(\Omega)^2$ if $\lambda = 0$). } Here $\nabla K$ is $\nabla G$ if $\lambda = 1$, and $\nabla L$ if $\lambda = 0$.

Note that when $\lambda = 0$, \eqref{velsolve2} reduces to the usual $H^1$ method, while the method with $\lambda = 1$ is motivated by the fact that the solution $\bth$ satisfies the divergence-free property, and thus the preservation of volume.

Since we consider domains which are $C^{1,1}$-regular, we can solve the mean curvature as $\kappa = \dive_{\Gamma_{\rm f}}{\bn} = \dive{\bN}$ \cite[Proposition 5.4.8]{henrot2018}, where $\bN$ is any extension of $\bn$ in $\Omega$. Having this in mind, we solve for $\bN\in H^1(\Omega)^2$ by means of the equation
\begin{align}
	\varepsilon a({\bN},{\bphi})_{\Omega} + ({\bN},{\bphi})_{\Gamma_{\rm f}} = ({\bn},{\bphi})_{\Gamma_{\rm f}} \ for\ all\ \bphi\in H^1(\Omega)^2.
	\label{normalextension}
\end{align} 
for sufficiently small $\varepsilon>0.$

\subsection{Deformation Algorithm}
Given an initial domain $\Omega_0$, we shall solve the minimization problem iteratively. In particular, we shall update an iterate $\Omega_k$ by the identity perturbation operator, i.e., $\Omega_{k+1} = (I + t_k{\bth}_k)\Omega_k$.  Here, ${\bth}_{k}$ is solved using \eqref{velsolve2} but on $\Omega_k$ instead of $\Omega$. On the other hand, we shall solve for $t_k$ using a line-search method starting at an initial guess
\begin{equation} 
t_k^0 = \beta\frac{\mathcal{J}(\Omega_k)}{\|\bth_k\|_{L^2(\Gamma_{\rm f})^2}^2}\text{ or }\beta\frac{\mathcal{L}(\Omega_k)}{\|\bth_k\|_{L^2(\Gamma_{\rm f})^2}^2}
\label{stepsize}
\end{equation}
where $\beta>0$ is a fixed, sufficiently small parameter.  This step size is updated whenever $\mathcal{J}(\Omega_{k+1}) > \mathcal{J}(\Omega_k)$ (or $\mathcal{L}(\Omega_{k+1}) > \mathcal{L}(\Omega_k)$) or if the deformed domain $\Omega_{k+1}$ exhibits mesh degeneracy (see \cite[Section 3.4]{dapogny2018} for more details).

As for the augmented Lagrangian $\mathcal{L}$, the multipliers $\ell$ and $b$ are iteratively updated as well. Inspired by Framework 17.3 in \cite{nocedal2006} for the $\ell$-update, and by \cite[Section 3.3]{dapogny2018}, we use the following rule
\begin{align}
	\ell_{n+1} = \ell_n - b_n\mathcal{F}(\Omega_k),\text{ and }b_{n+1}=\left\{\begin{aligned} \tau b_n &\text{ if }b_n<\overline{b}\\ b_n&\text{ otherwise,}\end{aligned}\right.
	\label{multiplier}
\end{align}
where $\overline{b}>0$ is a given target parameter and $\tau>1$ is an increase-iteration parameter.  One may refer to \cite[Section 3.3]{dapogny2018} for a more detailed rationale on the mentioned updates.

In summary, we have the following algorithm:\\
\begin{algorithm}[H]
  Initialize: $\Omega_0$, $\beta$, $\tau$, $\gamma$, $\ell_0$, and $b_0$\;
  Determine the mean curvature $\kappa = \dive \bN$ using \eqref{normalextension}\;
  Solve the respective states $({\bu}_0,p_0)$ and $({\bv}_0,\pi_0)$ from \eqref{operator} and  \eqref{wkadjoint} on $\Omega_0$\;
 \For{k=1,...,max iteration}{
 Solve the value of the objective function $\mathcal{J}(\Omega_k)$ ( or the Lagrangian $\mathcal{L}(\Omega_k)$)\;
 Solve the the deformation field ${\bth}_k\,$ from \eqref{velsolve2} on $\Omega_k$\;
 Set $\Omega_{k+1} = (Id + t_k{\bth}_k\,)\Omega_k$, given that $t_k$ is obtained from \eqref{stepsize} and if  $\mathcal{J}(\Omega_{k+1})<\mathcal{J}(\Omega_{k})$( or $\mathcal{L}(\Omega_{k+1})<\mathcal{L}(\Omega_{k})$ )\;
 Determine new mean curvature $\kappa$\;
 Solve new solutions $(\bu_{k+1},p_{k+1})$ and $(\bv_{k+1},q_{k+1})$\;
 \eIf{$|\mathcal{J}(\Omega_{k+1}) - \mathcal{J}(\Omega_{k+1})\text{( or }|\mathcal{L}(\Omega_{k+1}) - \mathcal{L}(\Omega_{k+1})| \text{)} < $tolerance}
 {break\;}
 {continue\;}
 Update $\ell_{k+1}$ and $b_{k+1}$ using \eqref{multiplier}\;
 }
 \caption{{\bf{\it aL}-algorithm} if $\lambda = 0$, {\bf{\it dF}-algorithm} if $\lambda = 1$}
\end{algorithm}

%

\subsection{Numerical Examples}
Using the algorithms discussed previously, we will now show some numerically implemented examples, which are all run using the programming software FreeFem++ \cite{hecht2012} on a  Intel Core i7 CPU $@$ 3.80 GHz with 64GB RAM. 

In all these examples, we shall consider a Poiseuille-like input function given by $\bg =(1.2(0.25 - x_2^2),0)^\top$. Furthermore, we consider the rectangular channel whose corners are $(0,-0.5)$, $(0,0.5)$, $(2,0.5)$, and $(2,-0.5)$. On the other hand we consider a circular initial shape for the free-boundary $\Gamma_{\rm f}$ that is centered at $(0.325,0)$ and with radius $r = 0.13$ (see Figure \ref{initialshape}). For simplicity, we shall also consider a fairly viscous fluid, i.e., $\nu = 1/100.$
\begin{figure}[h!]
 \centering
  \includegraphics[width=1\textwidth]{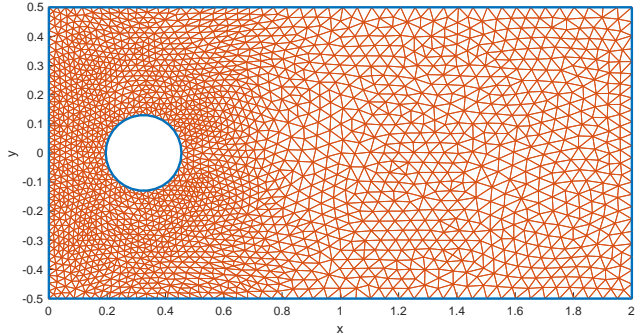}\vspace{-.1in}
 \caption{Initial geometry of the domain with the refined mesh.}
 \label{initialshape}
 \end{figure}

The volume constant $m$ is set as the volume of the initial domain $\Omega_0$, which is refined according to the magnitude of the state solution with the minimum mesh size $hmin = 1/50$ and a maximum mesh size $hmax = 1/30$. All the governing equations -- the state equation, adjoint equation, state responsible for the deformation fields, and the smooth approximation of the normal vector -- are solved using the UMFPACK option. The tolerance, furthermore, is selected to be $tol=1\times10^{-6}.$

\subsubsection{Augmented Lagrangian Method}
The implementation of {\bf{\it aL}-algorithm} is divided into two parts. Namely, we consider when $\gamma_1 = 1$ and $\gamma_2 = 0$ (which we call as the {\it curl}$_{aL}$-problem), $\gamma_1 = 0$ and $\gamma_2 = 1$ (which we call as the {\it detgrad}$_{aL}$-problem).

The values of the constants $\alpha$, $\ell_0$, $b_0$, $\tau$, and $\overline{b}$ are chosen so as to maintain stability of the method.\\

\noindent{\bf Implementation of {\it curl}$_{aL}$-problem} ($\gamma_1 = 1,\gamma_2=0$). For this implementation, we have chosen the numerical parameters shown in Table \ref{param:curlaL}.
\begin{table}[h!]
\caption{Parameter values for {\it curl}$_{aL}$-problem}
\centering
\begin{tabular}{|c c | c c|}
\hline\hline
Parameter 		& 		Value 						& 		Parameter 	& 		Value\\[0.5ex]
\hline\hline
$\alpha$			&		6.0							&		$\ell_0$		&		20\\
$b_0$				& 		$1\times10^{-4}$		&		$\tau$			&		1.05\\
$\overline{b}$	&		10						& &\\
\hline
\end{tabular}
\label{param:curlaL}
\end{table}

Using these parameters, the evolution of the free-boundary is shown in Figure \ref{auglcurlevo}({\it left}) below. We see that the free-boundary evolves into a bean-shaped surface.
\begin{figure}[h!]
 \hspace{-.5in}
  \includegraphics[width=.9\textwidth]{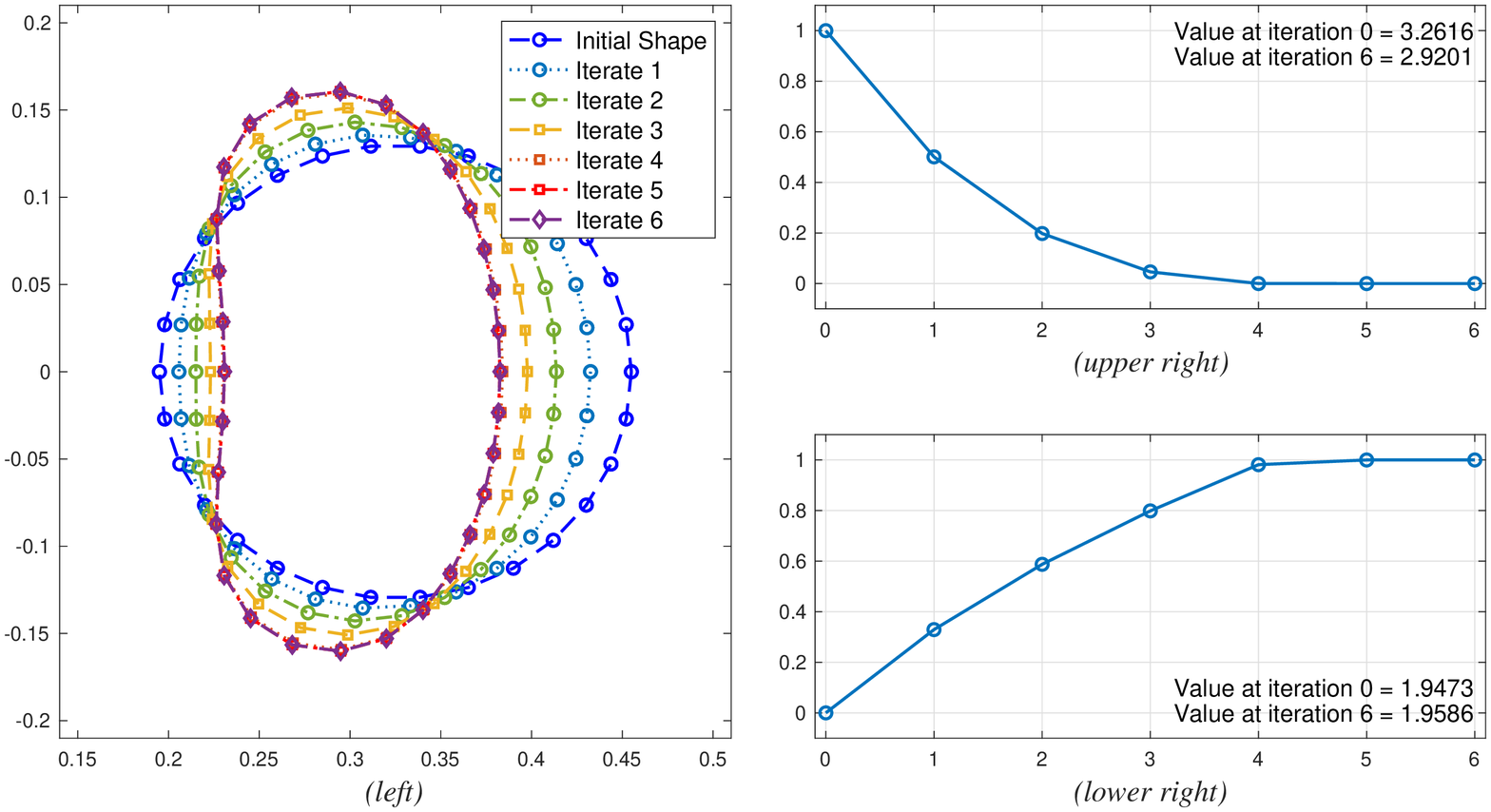}
 \caption{From {\it curl}$_{aL}$-problem, the figure features the following: ({\it left}) Evolution of the free-boundary $\Gamma_{\rm f}$, ({\it upper right}) Normalized trend of the objective functional, ({\it lower right}) Normalized trend of the volume}
 \label{auglcurlevo}
 \end{figure}
As expected, we see that the value of the objective functional decreases on each iteration (see Figure \ref{auglcurlevo}({\it upper right})). In fact, the value of the objective functional decreases by $10.47\%$. Lastly, we see from Figure \ref{auglcurlevo}({\it lower right}) that the volume increases by $0.58\%$.\\

\noindent{\bf Implementation of {\it detgrad}$_{aL}$-problem} ($\gamma_1= 0,\gamma_2 = 1$).  In this particular scenario, we consider the parameter values shown in Table \ref{param:detgradaL}.
\begin{table}[h!]
\caption{Parameter values for {\it detgrad}$_{aL}$-problem}
\centering
\begin{tabular}{|c c | c c|}
\hline\hline
Parameter 		& 		Value 						& 		Parameter 	& 		Value\\[0.5ex]
\hline\hline
$\alpha$			&		1.3							&		$\ell_0$		&		.5\\
$b_0$				& 		$1\times10^{-2}$		&		$\tau$			&		1.05\\
$\overline{b}$	&		10						& &\\
\hline
\end{tabular}
\label{param:detgradaL}
\end{table}

Unlike the {\it curl}$_{aL}$-problem, the shape does not evolve into a bean-shaped obstacle. The evolution acts in a non-symmetrical manner (see Figure \ref{augldetgradevo}({\it left})) due to the fact that the right hand-side of the adjoint equation with the current choice of $\gamma_1$ and $\gamma_2$ is also non-symmetric.  
Even though this problem runs seven iterations, it can be observed that after two iterations, the shapes are {\it almost} identical as shown in Figure \ref{augldetgradevo}({\it left}).  This phenomena can also be observed on the trend of the decrease of the value of the objective functional, i.e.,  the values of the functional from the second iteration up to the last iteration differ minimally. 
\begin{figure}[h!]
 \hspace{-.5in}
  \includegraphics[width=.9\textwidth]{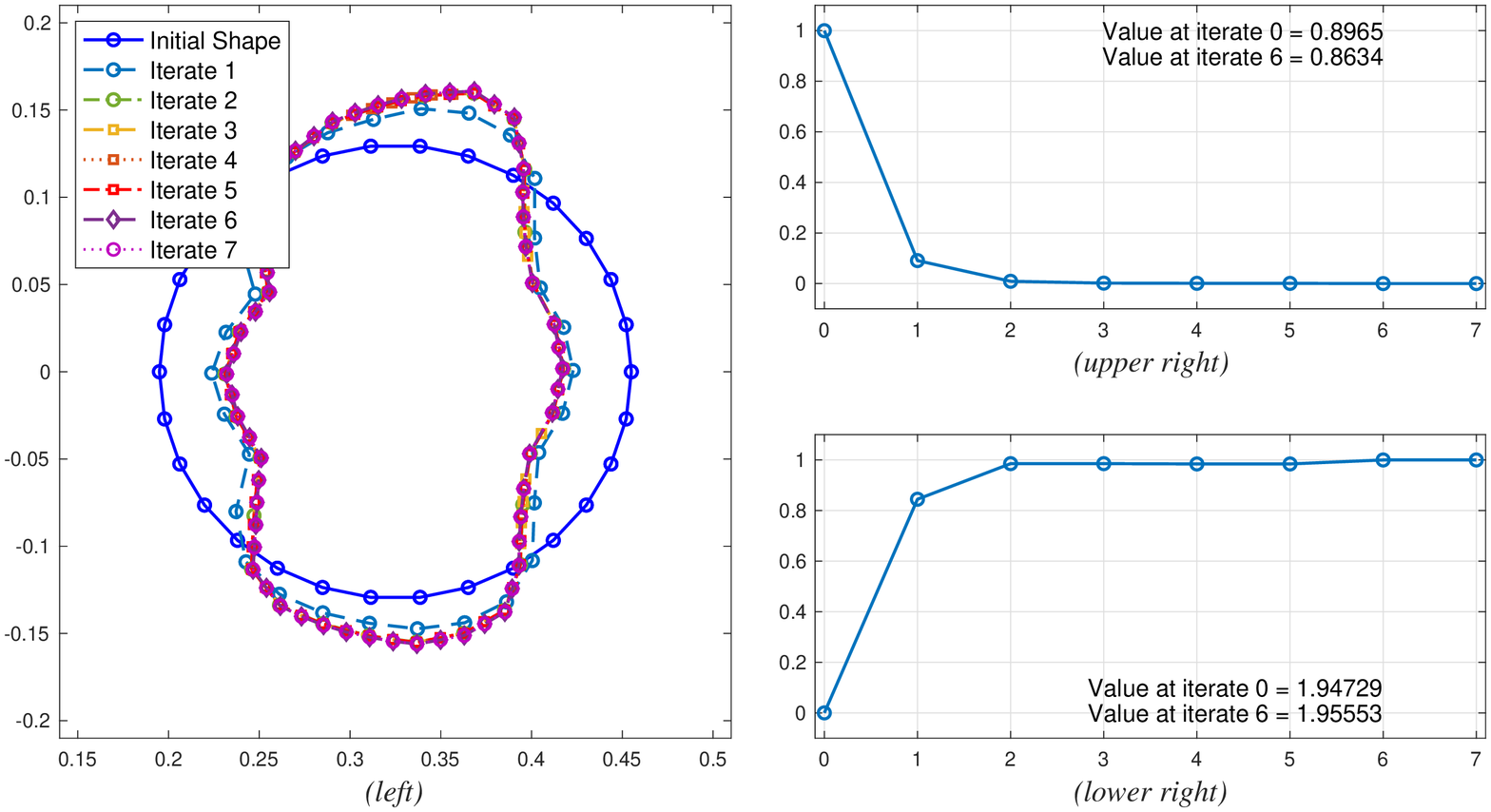}
 \caption{From {\it detgrad}$_{aL}$-problem, the figure features the following: ({\it left}) Evolution of the free-boundary $\Gamma_{\rm f}$, ({\it upper right}) Normalized trend of the objective functional, ({\it lower right}) Normalized trend of the volume}
 \label{augldetgradevo}
 \end{figure}
 
 The percentage change from the initial shape to the second iterate-shape is $3.67\%$, while from the second iterate to the final shape it is $0.03\%$. In general, from the initial shape to the final shape the percentage difference is $3.7\%.$ We also point out that the minute difference of iterate 2 and the final shape may be caused by the fact that the changes are caused by mesh adaptation used in our program.

\subsubsection{Divergence-free deformation fields}

Unlike the implementation of {\bf{\it aL}-algorithm}, the minimization problem is solved by {\bf{\it dF}-algorithm} in three parts. In particular, we shall illustrate the following modificatoins:
\begin{enumerate}
	\item[$\bullet$] {\it curl}$_{dF}$-problem with parameters $\gamma_1 = 1$, and $\gamma_2 = 0$;
	\item[$\bullet$] {\it detgrad}$_{dF}$-problem with parameters $\gamma_1 = 0$, and $\gamma_2 = 1$;
	\item[$\bullet$] {\it mixed}$_{dF}$-problem with parameters $\gamma_1 ,\gamma_2 \ge 1$.
\end{enumerate}

\noindent{\bf Implementation of {\it curl}$_{dF}$-problem} ($\gamma_1 = 1, \gamma_2 = 0$). 
Note that in using {\bf{\it dF}-algorithm}, we are more relaxed in the choice of the parameters. In particular, we are just concerned with the parameter $\alpha$ which we choose as $\alpha = 5$.

As we can see in Figure \ref{divfreecurlevo}({\it left}) the method converges to a bean shaped boundary, just like the generated shape in {\it curl}$_{aL}$-problem. However, in the current case, the method steers the shape to have the inner domain bounded by $\Gamma_{\rm f}$ lose its convexity on the left side.
\begin{figure}[h!]
 \hspace{-.5in}
  \includegraphics[width=.9\textwidth]{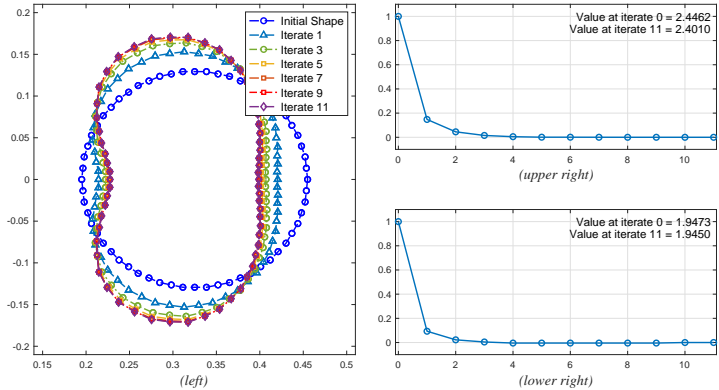}
 \caption{From {\it curl}$_{dF}$-problem, the figure features the following: ({\it left}) Evolution of the free-boundary $\Gamma_{\rm f}$, ({\it upper right}) Normalized trend of the objective functional, ({\it lower right}) Normalized trend of the volume}
 \label{divfreecurlevo}
 \end{figure}
One can also infer from Figure \ref{divfreecurlevo}({\it upper right}) that the relative difference - with respect to the initial shape - of the value of the objective function at the optimal shape is $1.85\%$, while of the volume is $0.12\%$. 

If we compare this algorithm with {\bf{\it aL}-algorithm} with $\alpha = 5$ and $\ell_0 = 13,13.5,14$, we deduce from Figure \ref{curlcompdivfwaugl}({\it lower right}) that our current method preserves the volume better. We also see from Figure \ref{curlcompdivfwaugl}({\it upper right})) that the objective function values using {\bf{\it dF}-algorithm} are lower than the ones from {\bf{\it aL}-algorithm} except from the case $\ell_0=14$, which fails horribly with volume preservation.
\begin{figure}[h!]
 \hspace{-.5in}
  \includegraphics[width=.9\textwidth]{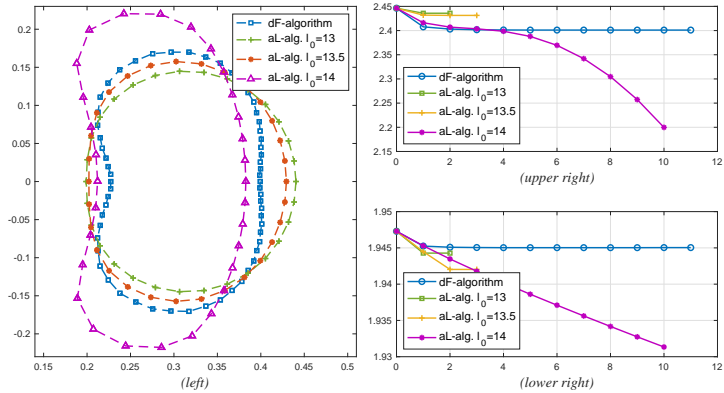}
 \caption{({\it left}) Comparison of final shapes between {\bf{\it aL}-algorithm} and {\bf{\it dF}-algorithm} for the problem with the parameters $\gamma_1 = 1$, $\gamma_2 = 0$ and $\alpha = 5$, ({\it upper right}) Comparison of objective value trends between {\bf{\it aL}-algorithm} and {\bf{\it dF}-algorithm}, ({\it lower right}) Comparison of volume trends between {\bf{\it aL}-algorithm} and {\bf{\it dF}-algorithm}}
 \label{curlcompdivfwaugl}
 \end{figure}

We also see in Figure \ref{curlcompdivfwaugl}({\it left}) what we meant we we said that the inner domain bounded by $\Gamma_{\rm f}$ loses its convexity on the right side, i.e., the curvature of the shapes generated by {\bf{\it aL}-algorithm} are more pronounced compared to our current algorithm.

\begin{remark}
(i.) The reason why we started the simulation of {\bf{\it aL}-algorithm} with $\alpha=5$ at $\ell_0 = 13$ is that if we try a smaller value of $\ell_0$, the scheme becomes too stringent and will not generate any domain perturbation.

\noindent(ii.) One interesting feature in Figure \ref{curlcompdivfwaugl} is that even after ten iterations, the objective function and volume trends for the case $\ell_0=14$ seem to have not converged yet. In fact, even after 20 iterations, the scheme still finds a descent direction. This highlights an issue with the balance between regularization and stability, which is an interesting venture to delve into if one is keen on studying stability of numerical methods. One may also refer to \cite{nair1997} for some insights on the issue of regularity and stability in general.
\end{remark}

\noindent{\bf Implementation of {\it detgrad}$_{dF}$-problem} ($\gamma_1 = 0, \gamma_2 = 1$). Similar with our previous implementation, we only needed to choose a value for the parameter $\alpha$, which in this case is $\alpha = 1.0.$

As illustrated in Figure \ref{divfreedetgradevo}({\it left}), the final perturbed domain differs minimally from the initial domain. Furthermore,  we can see from Figure \ref{detgradcompdivfwaugl}({\it left}) that the deformation is less extensive as compared to the final shapes generated by the {{\it aL}-algorithm}.  One can also infer from Figure \ref{divfreedetgradevo}({\it upper} and {\it lower right}) that the value of the objective function and of the volume at the final shape differs relatively (w.r.t. the initial shape) by $1.30\%$ and $0.03\%$, respectively. This implies that the volume preservation is significantly {\it satisfied} by the {\bf{\it dF}-algorithm}, while maintaining a significant decrease on the objective function.
\begin{figure}[h!]
 \hspace{-.5in}
  \includegraphics[width=.9\textwidth]{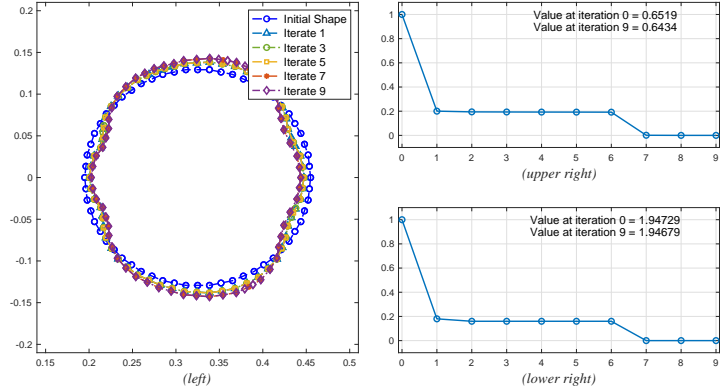}
 \caption{From {\it detgrad}$_{dF}$-problem, the figure features the following: ({\it left}) Evolution of the free-boundary $\Gamma_{\rm f}$, ({\it upper right}) Normalized trend of the objective functional, ({\it lower right}) Normalized trend of the volume}
 \label{divfreedetgradevo}
 \end{figure}

If we compare the current algorithm to {\bf{\it aL}-algorithm} with the same value of $\alpha$, we see from Figure \ref{detgradcompdivfwaugl}({\it lower right}) that the only value of $\ell_0$ that can preserve the initial volume better is $\ell_0=0.1$. We also see from Figure \ref{detgradcompdivfwaugl}({\it upper right}) that since the difference between the final value of the objective function at the final shape using {\bf{\it dF}-algorithm} and {\bf{\it aL}-algorithm} with $\ell_0=0.1$ is very minimal,  the {\bf{\it aL}-algorithm} is more preferable. The only drawback though is that {\bf{\it aL}-algorithm} depends on more parameters, hence it is more subject to instabilities.
\begin{figure}[h!]
 \hspace{-.5in}
  \includegraphics[width=.9\textwidth]{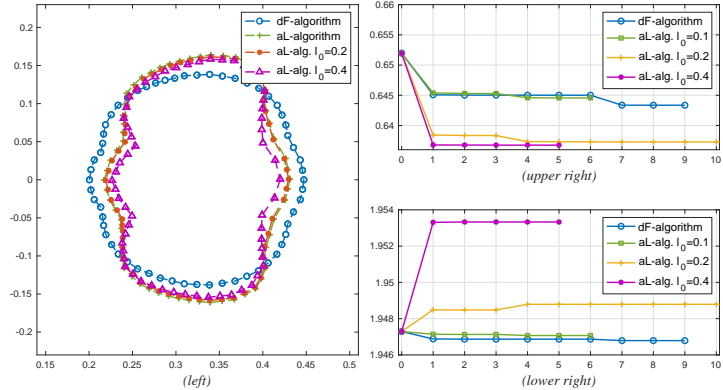}
 \caption{({\it left}) Comparison of final shapes between {\bf{\it aL}-algorithm} and {\bf{\it dF}-algorithm}  with the parameters $\gamma_1 = 0$, $\gamma_2 = 1$, and $\alpha = 1$, ({\it upper right}) Comparison of objective value trends between {\bf{\it aL}-algorithm} and {\bf{\it dF}-algorithm}, ({\it lower right}) Comparison of volume trends between {\bf{\it aL}-algorithm} and {\bf{\it dF}-algorithm}}
 \label{detgradcompdivfwaugl}
 \end{figure}

The implementations of the {\bf{\it aL}-algorithm} with $\ell_0=0.2,0.4$ exhibit a significant decrease on the objective functional as compared to the {\bf{\it dF}-algorithm}, however they both violate the volume constraint significantly.

To illustrate another case where the {\bf{\it aL}-algorithm} is more preferable for the minimization of the {\it detgrad} functional,  Figure \ref{detgradcompdivfwaugl2} shows an implementation where $\alpha = 1.2$, and $\ell_0 = 0.1$ for the {\bf{\it aL}-algorithm}.
\begin{figure}[h!]
 \hspace{-.5in}
  \includegraphics[width=.9\textwidth]{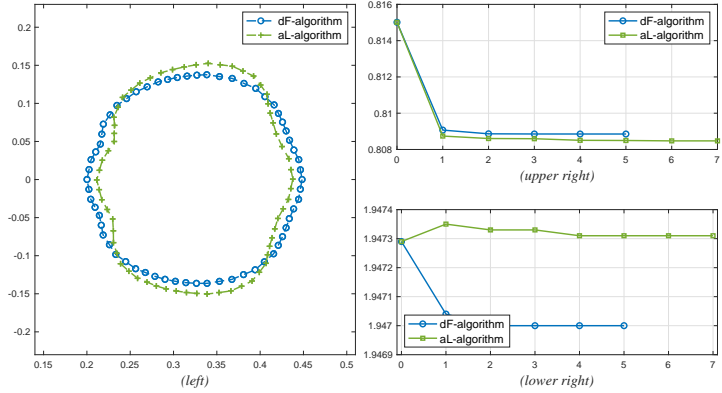}
 \caption{({\it left}) Comparison of final shapes between {\bf{\it aL}-algorithm} and {\bf{\it dF}-algorithm} with the parameters $\gamma_1 = 0$, $\gamma_2 = 1$, and $\alpha = 1.2$, ({\it upper right}) Comparison of objective value trends between {\bf{\it aL}-algorithm} and {\bf{\it dF}-algorithm}, ({\it lower right}) Comparison of volume trends between {\bf{\it aL}-algorithm} and {\bf{\it dF}-algorithm}}
 \label{detgradcompdivfwaugl2}
 \end{figure}

As can be easily observed in Figure \ref{detgradcompdivfwaugl2}({\it right}), {\bf{\it aL}-algorithm} yields better result when it comes to the minimization of the objective functional and preservation of the volume. In particular, the objective functional values are much lower for {\bf{\it aL}-algorithm}, and the volume difference between the initial and final shape are much closer compared to the values generated by {\bf{\it dF}-algorithm}.\\

\noindent{\bf Implementation of {\it mixed}$_{dF}$-problem} ($\gamma_1,\gamma_2\ge 1$). In this implementation, we consider several configurations for reasons we shall explain later.  In particular,  we shall consider the value of parameters as shown in Table \ref{divfmixedsetup}.
 \begin{table}[h!]
\caption{Parameter Values for {\it mixed}$_{dF}$-problem}
\centering
\begin{tabular}{|c| c  c c|}
\hline\hline
Configuration 		& 		$\alpha$ 					& 		$\gamma_1$ 	& $\gamma_2$ \\[0.5ex]
\hline\hline
1							&		6.0							&		1.0					&		1.0\\
2							&		7.0							&		1.0					&		2.0\\
3							&		8.0							&		1.0					&		3.0\\
4							&		9.0							&		1.0					&		4.0\\
5							&		10.0							&		1.0					&		5.0\\
6							&		11.0							&		1.0					&		6.0\\
7							&		12.0							&		1.0					&		7.0\\
8							&		13.0							&		1.0					&		8.0\\
9							&		14.0							&		1.0					&		9.0\\
10							&		15.0							&		1.0					&		10.0\\
\hline
\end{tabular}
\label{divfmixedsetup}
\end{table}

For the first configuration,  Figure \ref{divfreemixedevo} shows how the initial shape evolves into a bean shaped boundary similar with that of {\it curl}$_{dF}$-problem. Furthermore, the value of the objective functional at the final shape differs relatively from the initial shape by $0.87\%$ while the volume is preserved with a relative difference of $0.30\%$.
\begin{figure}[h!]
\hspace{-.5in}
  \includegraphics[width=.9\textwidth]{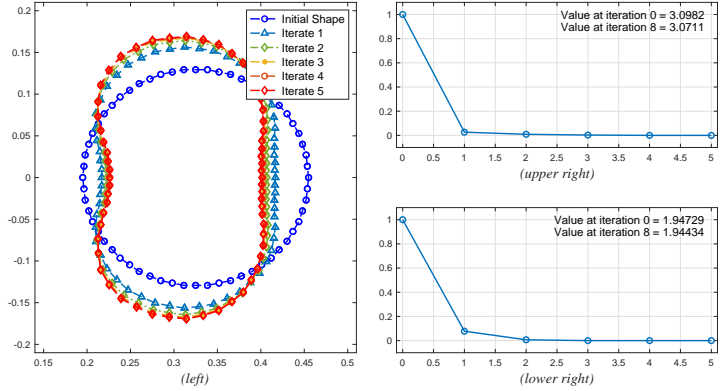}
 \caption{From configuration 1 of {\it mixed}$_{dF}$-problem, the figure features the following: ({\it left}) Evolution of the free-boundary $\Gamma_{\rm f}$, ({\it upper right}) Normalized trend of the objective functional, ({\it lower right}) Normalized trend of the volume}
 \label{divfreemixedevo}
 \end{figure}

We note that in this configuration the final shape is almost the same with the final shape generated by {\it curl}$_{dF}$-problem, as shown in Figure \ref{divfreemixedparted}({\it left}). 
\begin{figure}[h!]
 \hspace{-.5in}
  \includegraphics[width=.9\textwidth]{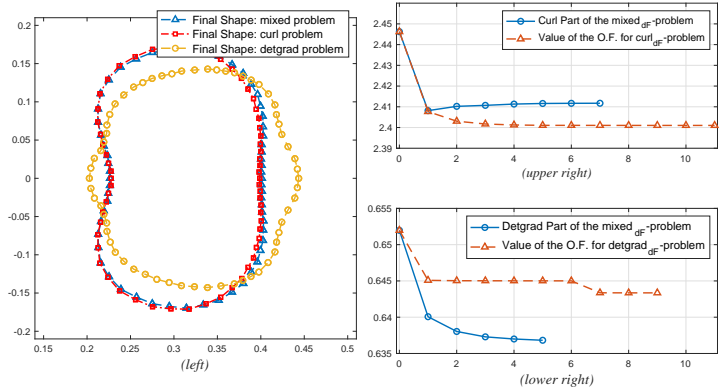}
 \caption{({\it left}) Plots of the solutions of {\it mixed}$_{dF}$-problem (using configuration 1),  of {\it curl}$_{dF}$-problem and of {\it detgrad}$_{dF}$-problem generated boundaries, ({\it upper right}) Comparison of the objective function value of the {\it curl}$_{df}$-problem and the {\it curl} part of the {\it mixed}$_{dF}$-problem, ({\it lower right}) Comparison of the objective function value of the {\it curl}$_{dF}$-problem and the {\it detgrad} part of the {\it mixed}$_{df}$-problem}
 \label{divfreemixedparted}
 \end{figure}
To understand why this phenomena occurs, we note that the objective function in our configurations can be written as
\[ \mathcal{J}(\Omega) = \left[5\int_{\Gamma_{\rm f}}\du s - \frac{1}{2}\int_{\Omega} |\nabla\times {\bu}|\du x \right] + k\left[\int_{\Gamma_{\rm f}}\du s - \int_{\Omega} h(\mathrm{det}(\nabla{\bu})) \du x  \right], \]
where $k$ is the configuration number. We shall call the first bracket as the {\it curl} part of the {\it mixed}$_{df}$-problem while the other one the {\it detgrad} part. Evaluating the two parts numerically at the initial shape, we can see in Figure \ref{divfreemixedparted}({\it right}) that the value of the {\it curl} part is higher than the {\it detgrad} part, with the respective values $2.45$ and $0.65$. This results to having the optimization problem prioritize minimizing the {\it curl} part of the {\it mixed}$_{dF}$-problem, and this results into having the similar shape as that of the {\it curl}$_{dF}$-problem. Intriguingly though, when compared to the values of the objective functional for the {\it curl}$_{dF}$-problem and the {\it detgrad}$_{dF}$-problem, the values of the {\it curl} part on each iterate is much higher than the objective function for {\it curl}$_{dF}$-problem (see Figure \ref{divfreemixedparted}({\it upper right})) while the values of the {\it detgrad} part is lower than the objective function for {\it detgrad}$_{dF}$-problem (see Figure \ref{divfreemixedparted}({\it lower right})).

To investigate the influence of the {\it curl}$_{dF}$-problem and the {\it detgrad}$_{dF}$-problem to the solution obtained by {\it mixed}$_{dF}$-problem even further, we implement configuration 2. In this example, the value of the {\it detgrad} part is higher than the one considered on the previous configuration.

\begin{figure}[h!]
 \hspace{-.5in}
  \includegraphics[width=.9\textwidth]{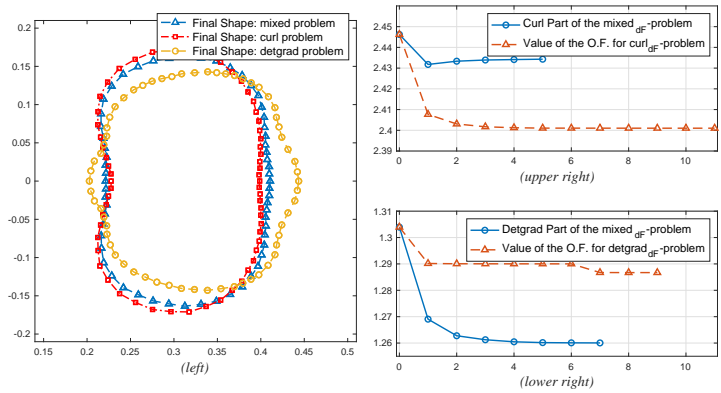}
 \caption{({\it left}) Plots of the solutions of {\it mixed}$_{dF}$-problem (using configuration 2),  of {\it curl}$_{dF}$-problem and of {\it detgrad}$_{dF}$-problem generated boundaries, ({\it upper right}) Comparison of the objective function value of the {\it curl}$_{dF}$-problem and the {\it detgrad} part of the {\it mixed}$_{dF}$-problem, ({\it lower right}) Comparison of the objective function value of the {\it curl}$_{dF}$-problem and the {\it detgrad} part of the {\it mixed}$_{dF}$-problem}
 \label{divfreemixedparted2}
 \end{figure}
As can be observed on Figure \ref{divfreemixedparted2}({\it left}) the  final shape for the {\it mixed}$_{dF}$-problem is now more distinct from that of the {\it curl}$_{dF}$-problem.  In particular,  we now see a bump emerging on the right side of the final shape for {\it mixed}$_{dF}$-problem. The emergence of such bump may be attributed to its presence on the final shape of {\it detgrad}$_{dF}$-problem. In fact, the value of the {\it detgrad} part is now higher compared to the previous configuration (see Figure \ref{divfreemixedparted2}({\it upper right}) and ({\it lower right})).

Skipping to the fourth configuration, we see that value of the {\it detgrad} part is higher than the {\it curl} part.  From Figure \ref{divfreemixedparted3}({\it left}), the final shape of the {\it mixed}$_{df}$-problem is now greatly influenced by the shape of {\it detgrad}$_{df}$-problem. In particular,  it can be observed that the right side of the shape generated by {\it mixed}$_{df}$-problem {\it converges} to the shape induced by the {\it detgrad}$_{df}$-problem. However, we can still observe the effect of the shape by the {\it curl}$_{df}$-problem on the left side. 

\begin{figure}[h!]
 \hspace{-.5in}
  \includegraphics[width=.9\textwidth]{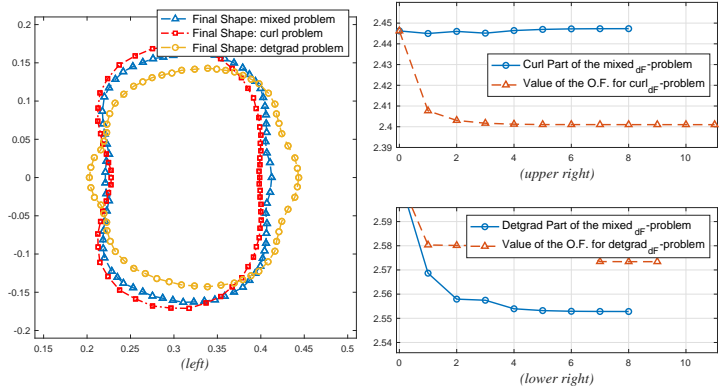}
 \caption{({\it left}) Plots of the solutions of {\it mixed}$_{dF}$-problem (using configuration 4),  of {\it curl}$_{dF}$-problem and of {\it detgrad}$_{dF}$-problem generated boundaries, ({\it upper right}) Comparison of the objective function value of the {\it curl}$_{dF}$-problem and the {\it detgrad} part of the {\it mixed}$_{dF}$-problem, ({\it lower right}) Comparison of the objective function value of the {\it curl}$_{dF}$-problem and the {\it detgrad} part of the {\it mixed}$_{dF}$-problem}
 \label{divfreemixedparted3}
 \end{figure}


To see the convergence more clearly we utilize Hausdorff distance. Figure \ref{hausdorff}({\it upper right}) illustrates the increase of the Hausdorff distance between the solutions of the {\it mixed}$_{dF}$-problem and the {\it curl}$_{dF}$-problem.  Meanwhile, Figure \ref{hausdorff}({\it lower right}) shows a decreasing trend in the Hausdorff distance between the solutions of the {\it mixed}$_{dF}$-problem and the {\it detgrad}$_{dF}$-problem.
\begin{figure}[h!]
 \hspace{-.5in}
  \includegraphics[width=.9\textwidth]{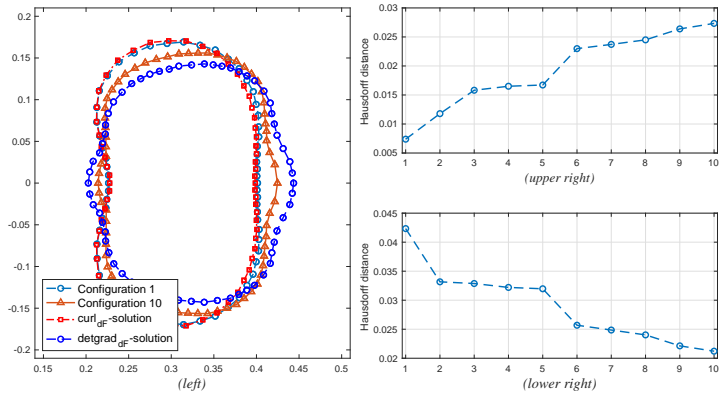}
 \caption{({\it left}) Plots of the solutions of {\it mixed}$_{dF}$-problem (configurations 1 and 10), of {\it curl}$_{dF}$-problem, and of {\it detgrad}$_{dF}$-problem, ({\it upper right}) Hausdorff distance between {\it mixed}$_{dF}$-solution and {\it curl}$_{dF}$-solution on each configuration, ({\it lower right}) Hausdorff distance between {\it mixed}$_{dF}$-solution and {\it detgrad}$_{dF}$-solution on each configuration}
 \label{hausdorff}
 \end{figure}

\subsection{Effects of the Shape Solutions to Generation of Twin-Vortex}

We conclude this section by briefly illustrating the effects of the final shapes to twin-vortices induced by a time-dependent Navier--Stokes equations right before the shedding of Karman vortex, and is solved using a stabilized Lagrange--Galerkin method \cite{notsu2016}. Since similar behaviors can be observed if we use {\bf{\it aL}-algorithm} or {\bf{\it dF}-algorithm}, we only illustrate the cases using {\bf{\it dF}-algorithm}. 

Figure \ref{initxcurl} shows the comparison of the twin-vortices using the initial domain and the domain obtained from {\it curl}$_{dF}$-problem. The figure shows that the solution to the {\it curl}$_{dF}$-problem generates a longer vortex compared to that of the initial domain.

\begin{figure}[h!]
 \centering\vspace{-.1in}
  \includegraphics[width=1\textwidth]{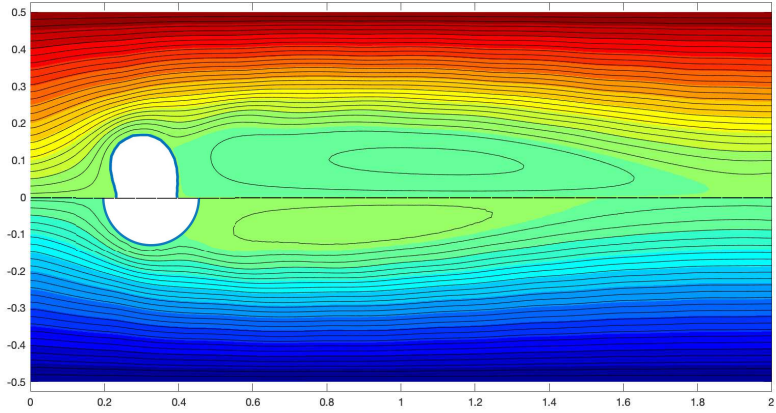}\vspace{-.2in}
 \caption{The figure shows the comparison of the flows using the initial domain (lower part), and the final shape from {\it curl}$_{df}$-problem (upper part)}
 \label{initxcurl}
 \end{figure}

Similarly, the solution of the {\it detgrad}$_{dF}$-problem generates a longer vortex as compared to the initial domain as shown in Figure \ref{initxcurl}.
\begin{figure}[h!]
 \centering\vspace{-.1in}
  \includegraphics[width=1\textwidth]{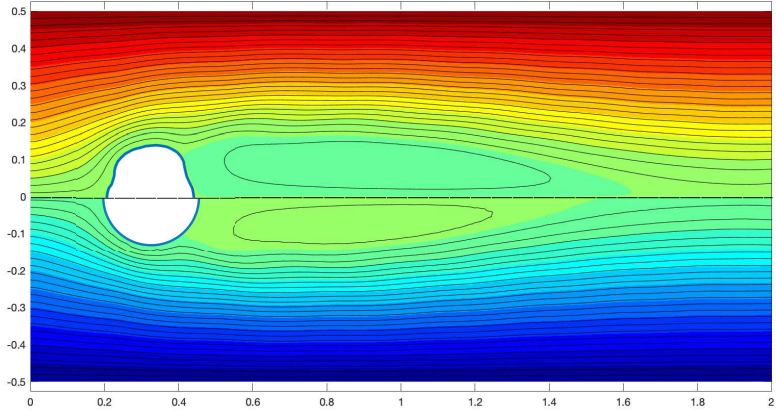}\vspace{-.2in}
 \caption{The figure shows the comparison of the flows using the initial domain (lower part), and the final shape from {\it detgrad}$_{df}$-problem (upper part)}
 \label{initxdetgrad}
 \end{figure}
  
  Lastly, Figure \ref{curlxdetgrad} shows that the vortex generated using the final shape of {\it curl}$_{df}$-problem is longer than that of the {\it detgrad}$_{df}$-problem.
\begin{figure}[h!]
 \centering\vspace{-.1in}
  \includegraphics[width=1\textwidth]{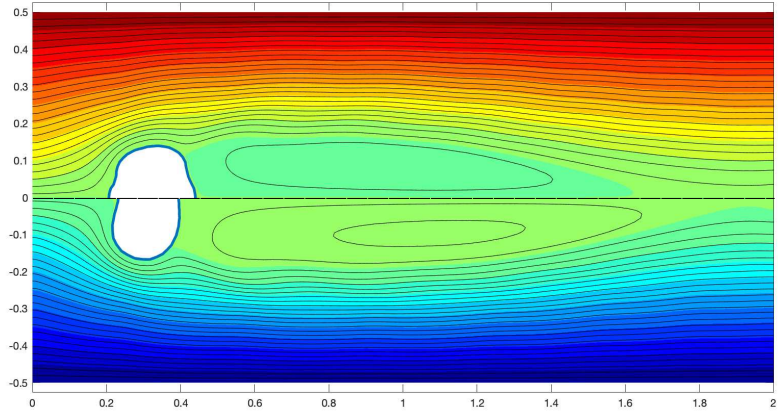}\vspace{-.2in}
 \caption{The figure shows the comparison of the flows using the the final shape from {\it curl}$_{df}$-problem (lower part), and the final shape from {\it detgrad}$_{df}$-problem (upper part)}
 \label{curlxdetgrad}
 \end{figure} 

\begin{remark}
As mentioned before, the generation of vortices gives a possible application of our shape optimization problem to a branch of machine learning known as physical reservoir computing. According to Goto, K., et.al.\cite{goto2021}, the length or the long diameter of the twin-vortices has a correlation with the memory capacity - as defined in the aforementioned reference - of the physical reservoir computer.
With our previous observations, we can {\it naively} conjecture that the shapes generated by the implementations above will cause a better memory capacity for the physical reservoir computer.
 However, since these are just visual observations, we need more experiments to verify that these shapes indeed generate a good computational ability.
 \end{remark}
\section{Conclusion}

We presented a minimization problem that is intended to maximize the vorticity of the fluid governed by the Stokes equations.  Furthermore, we considered a Tikhonov-type regularization - in the form of the perimeter functional - and a volume constraint. 
The shape sensitivity analysis is carried out by utilizing the rearrangement method which gave us the necessary conditions  in the form of the Zolesio--Hadamard Structure. We then implemented the necessary conditions in two ways, one is by utilizing the augmented Lagrangian, and by utilizing a new method for solving the deformation field that possesses the divergence-free condition. 
From these methods, we showed some numerical examples. Lastly, we illustrated how the shape solutions affect the flow, and we briefly mentioned a possible application of this problem in the field of physical reservoir computers.

We end this note by pointing out that the investigation in the stability of the numerical approaches with respect to the parameters were only scratched on the surface and can be explored with more rigor. Furthermore, the visual observations on the emergence of the twin-vortex are not enough to give a good conclusion with regards to the memory capacity of the physical reservoir computer. With this, one can investigate even further and try some more tests as carried out in \cite{goto2021}.


\end{document}